%% file: main_teorico.tex
\documentclass[10pt]{amsart}
\usepackage{amsmath}
\usepackage{amsfonts} 
\usepackage{bbm}
\usepackage{mathrsfs}
\usepackage{xifthen} 
\usepackage{hyperref}
\usepackage{cleveref}
\usepackage[numbers,sort&compress]{natbib}
\usepackage{bigints}
\usepackage{graphicx}
\usepackage{booktabs}
\usepackage{para list}
\usepackage[shortlabels]{enumitem}
\usepackage{tikz}
\usetikzlibrary{tikzmark}
\input{newcommands}
\renewcommand{\paragraph}[1]{\subsubsection*{#1}}

\usepackage{setspace}
\usepackage[a4paper,left=3.0cm,right=3.0cm]{geometry}
\title{Convergence analysis of the intrinsic surface finite element method}
\author{Elena Bachini$^1$ and Mario Putti$^2$}
\address{$^1$ Department of Mathematics ``Tullio Levi-Civita'', University
    of Padua, Italy  
   (elena.bachini@unipd.it)\\
   $^2$ Department of Agronomy, Food, Natural
    resources, Animals and Environment, University of Padua, Italy
    (mario.putti@unipd.it)
}
\begin{document}
\maketitle
\begin{abstract}
  The Intrinsic Surface Finite Element Method (ISFEM) was recently
  proposed to solve Partial Differential Equations (PDEs) on surfaces.
  ISFEM proceeds by writing the PDE with respect to a local coordinate
  system anchored to the surface and makes direct use of the resulting
  covariant basis.  Starting from a shape-regular triangulation of the
  surface, existence of a local parametrization for each triangle is
  exploited to approximate relevant quantities on the local
  chart. Standard two-dimensional FEM techniques in combination with
  surface quadrature rules complete the ISFEM formulation thus
  achieving a method that is fully intrinsic to the surface and makes
  limited use of the surface embedding only for the definition of
  basis functions. However, theoretical properties have not
  yet been proved. In this work we complement the original derivation
  of ISFEM with its complete convergence theory and propose the
  analysis of the stability and error estimates by carefully tracking
  the role of the geometric quantities in the constants of the error
  inequalities.  Numerical experiments are included to support the
  theoretical results.
\end{abstract}
 



\section{Introduction}
\label{sec:intro}
Surface phenomena are ubiquitous in nature, playing an important role
in mediating exchanges between contrasting media. They encompass a
wide range of scales, from nano to planetary, with examples including
earth processes~\cite{art:Bouchut2004, art:FernandezNieto2008,
  art:Flyer2012, art:Ferroni2016, art:Fent2017, art:BP20}, biological
applications~\cite{ art:Neilson2011, art:Lowengrub2016,
  art:nitschke2012}, and image
processing~\cite{art:Bertalmo2001,art:Stocker2008}.  These models are
typically based on partial differential equations (PDEs) governing
balance laws of scalar, vector, and tensor quantities living on the
surface. The detailed mathematical understanding of these PDEs is
still limited, and applications are tackled by numerical techniques.

Most of the approaches developed so far for the discretization of
surface PDEs rely heavily on the embedding in the ambient Euclidean
space to project quantities back to the surface. In essence, the
quantities of interest arising from the solution of the PDE are
extended to a tubular neighborhood of the surface and then projected
back to the surface or its piecewise interpolation, thus avoiding
altogether the need to use charts~\cite{art:Nestler2019}. This
strategy has allowed the development of conforming and nonconforming
finite element methods using the so-called Surface Finite Element
(SFEM) originally developed in~\cite{art:Dziuk1988}
(see~\cite{art:Dziuk2013} for a recent review), with extensions to
discontinuous Galerkin~\cite{art:Antonietti2015} and low order virtual
element methods~\cite{art:Frittelli2018}.
A different approach has been recently proposed in~\cite{art:BFP21},
where the Intrinsic Surface Finite Element Method (ISFEM) has been
developed and favorably compared to the embedded approach
of~\cite{art:Dziuk2013} for the surface advection-diffusion-reaction
equation, and to other available methods for the scalar and
vector-valued surface heat equation~\cite{art:Praetorius2024}. A
variant of ISFEM has been used to develop arbitrary order virtual
elements on surfaces~\cite{art:BMP21}, by working directly on the
chart, and thus needs a complete knowledge of the parametrization.

The ISFEM method is based on piecewise linear approximations of the
discrete spaces and relies exclusively on geometric quantities that
are intrinsic to the surface.  The main advantage of ISFEM with
respect to other embedded approaches is that the numerical solution,
whether scalar, vector, or tensor, is an object intrinsically defined
on the surface, avoiding the need to define its extension in $\REAL^3$
and its projection back to the surface. In addition, the ISFEM
formulation in the scalar case requires only the knowledge of the
tangent plane at the vertices of the surface triangulation in a form
that can be exact, by means of the knowledge of the parametrization,
or approximate, i.e., starting from point data. The embedding of the
surface $\SurfDomain$ in $\REAL^3$ is used exclusively in the
definition of the basis functions, which are defined by lifting onto
local tangent planes the first order polynomials defined either in the
ambient space or on a local chart. As a consequence, the
method can freely use different embeddings (metrics), can be adapted
to multiple charts if available, and, unlike SFEM or other approaches
(see e.g.~\cite{art:Nestler2019,art:Praetorius2024}), can be
extended with few and straight-forward modifications to vector and
tensor-valued PDEs (see~\cite{art:Praetorius2024} for the
vector-valued case).

In this paper we develop the full numerical analysis of ISFEM, not yet
addressed in previous work.  It turns out that the convergence
estimates arise directly from the redefinition of an appropriate
scalar product intrinsic to $\SurfDomain$.
We start with the identification of a proper Local Coordinate System
(LCS) anchored on the surface $\SurfDomain$, we discuss the weak
formulation of the PDE written in covariant form and defined on
appropriate surface Sobolev spaces, e.g., $\Sob{1}(\SurfDomain)$ for a
closed surface without boundary or $\Sob{1}_0(\SurfDomain)$ for a
surface with homogeneous Dirichlet boundary. This operation introduces
anisotropy due to the presence of the metric tensor arising from the
first fundamental form of $\SurfDomain$. This anisotropy, whose ratio
remains always bounded for a regular surface, adds to the eventual
anisotropy of a tensor-valued diffusion coefficient.  This added
difficulty is counterbalanced by the fact that the ensuing numerical
discretization, being defined on the LCS and thus on the chart or
atlas, can exploit all the techniques developed for a planar
two-dimensional domain and inherits all the related properties.
For this reason, in our convergence estimates we discuss how the
inequality constants depend upon the surface geometric quantities,
intrinsic or extrinsic, i.e., depending on the first or second
fundamental form. Our calculations show that optimal second order
convergence is obtained under standard assumptions on the regularity
of the mesh.  This is experimentally discussed in the numerical
results section, where convergence rates with respect to a
manufactured solution are exposed and the behavior of the error
constants at varying curvatures discussed. An example on a sphere
defined by multiple charts is also presented.

\section{The Intrinsic Surface FEM}
\label{sec:intrinsic}

Consider a compact surface $\SurfDomain\subset\REAL[3]$ over which we
would like to solve an elliptic equation of the form:
\begin{equation}
  \label{eq:pde}
  -\DivSurf\left(\DiffTens\GradSurf\Sol\right)
  =\force \qquad \text{on} \quad \SurfDomain
  \, ,
\end{equation}
where $\SurfDomain$ is assumed to be fixed in time and the solution
$\Sol\From\SurfDomain\rightarrow\REAL$ is a scalar function defined on
the surface.  The tensor $\DiffTens$ is a rank-2 symmetric and
positive-definite diffusion tensor, and we assume 
$\force\in L^2(\SurfDomain)$. The differential operators $\DivSurf$
and $\GradSurf$, the surface divergence and gradient, respectively,
need to be properly defined to follow the geometric setting of the
problem.
If the compact surface has no boundary,~\cref{eq:pde} is augmented
by the constraints of zero mean on $\Sol$ and $\force$.  If the
surface has boundary (i.e., $\SurfDomainBnd\ne\{\emptyset\}$) we
assume zero Neumann conditions, again implicitly augmented by the
zero mean constraints,
or zero Dirichlet boundary conditions. For the handling
of non-homogeneous boundary conditions we refer
to~\citep{art:Burman2018}.

\subsection{Geometrical setting}
\label{sec:intrinsic-operators}
Let $\SurfDomain\subset\REAL[3]$ be a $2$-dimensional $\Cont[k]$
regular surface. 
We recall that a surface $\SurfDomain\subset\REAL[3]$ is said to be
$\Cont[k]$ regular if for any point $\point\in\SurfDomain$
there exists a map $\MapU[\point]:\SubsetU\To\REAL[3]$ of class
$\Cont[k]$,
with $\SubsetU\subset\REAL[2]$, such that
$\MapU[\point](\SubsetU)\subset\SurfDomain$ is a neighborhood of
$\point$, i.e., there exists an open neighborhood
$\SubsetV\subset\REAL[3]$ of $\point$ for which
$\MapU[\point](\SubsetU)=\SubsetV\cap\Gamma$, and such that
$\MapU[\point]$ is a diffeomorphism of its image. The map
$\MapU[\point]$ is called a local parametrization centered in
$\point$.
The inverse of the parametrization,
$\MapU[\point]^{-1}:\SubsetV\cap\SurfDomain\To\SubsetU$,
is called a local chart in $\point$. 
Explicitly, we have the following transformations:
\begin{align*}
  \MapU[\point]\From\SubsetU&\To\SubsetV\cap\SurfDomain
  &  \MapU[\point]^{-1}\From\SubsetV\cap\SurfDomain&\To\SubsetU\\
  \hxv & \Mapsto \xv
  &  
  \xv& \Mapsto \hxv
\end{align*}
where $\hxv=(\hgcscomp)$ are the local coordinates,
$\MapU[\point](\SubsetU)$ is the coordinate
neighborhood, and $\xv=(\gcscomp)$ are the global Cartesian
coordinates of a point on the surface.
Given two points $\point,\qpoint\in\SurfDomain$ and their local
parametrizations $\MapU[\point],\MapU[\qpoint]$, with
$\SubsetU_{\point}\cap\SubsetU_{\qpoint}\ne\emptyset$, we say that the
local parametrizations are compatible if the transition map
$\MapU[\point]\circ\MapU[\qpoint]^{-1}$ is a $\Cont[k]$
diffeomorphism.
We assume to have a family $\mathcal{A}=\{\MapU_{\alpha}\}$ of
compatible local parameterizations
$\MapU_{\alpha}:\SubsetU_{\alpha}\To\SurfDomain$ such that
$\SurfDomain=\cup_{\alpha}\MapU_{\alpha}(\SubsetU_{\alpha})$ (an
atlas for $\SurfDomain$).

Given a point $\point\in\SurfDomain$, the practical construction of
the relevant objects proceeds as
follows~\citep{art:BP20,art:BFP21}. We calculate the two tangent
vectors $\hvecBaseCCcv[1](\point)$ and $\hvecBaseCCcv[2](\point)$ on
$\TanPlane{\SurfDomain}$:
\begin{equation*}
  \hvecBaseCCcv[i] (\point) = 
  \Diff[\point]{\MapU}(\vecBaseGC[j] (\point)) = 
  \left(
    \DerPar{\xcg}{\hxvcomp[i]}, \DerPar{\ycg}{\hxvcomp[i]},
    \DerPar{\zcg}{\hxvcomp[i]} 
  \right), \qquad i = 1,2 \text{ and } j=1,2,3,
\end{equation*} 
where $\Diff[\point]{\MapU}$ is the Jacobian matrix of the coordinate
transformation and $\vecBaseGC[1](\point)$, $\vecBaseGC[2](\point)$,
$\vecBaseGC[3](\point)$ are the canonical basis vectors of $\REAL[3]$.
For numerical stability, we orthogonalize via Gram-Schmidt the vector
$\hvecBaseCCcv[2]$ with respect to $\hvecBaseCCcv[1]$, yielding the
LCS orthogonal frame $\vecBaseCCcv[1],\vecBaseCCcv[2]$ on
$\TanPlane[\point]{\SurfDomain}$. We will denote by $\sv=(\lcscomp)$
the corresponding local coordinates.
We denote with $\hJac[\point]=\left[ \hvecBaseCCcv[1] (\point),
  \hvecBaseCCcv[2](\point)\right]$ the Jacobian matrix of the
coordinate transformation between $\hxv$ and $\xv$, i.e.,
$\Diff[\point]{\MapU}$, and collect the orthogonal reference
vectors into the matrix $\Jac[\point]=\left[ \vecBaseCCcv[1]
  (\point), \vecBaseCCcv[2](\point)\right]$.
Note that, $\Jac[\point]$ is the Jacobian of the (unknown) coordinate
transformation $\MapV[\point]$ between $\sv$ and $\xv$, which exists
due to the regularity of the surface.
Given the above definitions, the unknown relation between the $\hxv$
and $\sv$ coordinates can be formally written as
$\sv=\MapV[\point]^{-1}\circ\MapU[\point](\hxv)$. On the other hand,
its Jacobian is know explicitly and can be written as:
\begin{equation*}
  \matW[\point]:=\invJac[\point]\hJac[\point]\in\REAL[2\times 2]\,,
\end{equation*}
where $M^+=(M^{T}M)^{-1}M^{T}$ is the pseudoinverse of the matrix
$M\in \REAL[3\times 2]$.
\begin{remark}
  Note that, $\matW[\point]$ is related to the Gram-Schmidt
  process that transforms the reference vectors
  $\{\hvecBaseCCcv[1](\point),\hvecBaseCCcv[2](\point)\}$ into the
  orthogonal frame
  $\{\vecBaseCCcv[1](\point),\vecBaseCCcv[2](\point)\}$.
\end{remark}

In this setting, the metric tensor is diagonal and is given by:
\begin{equation}
  \label{eq:first-FF}
  \First(\point) :=
  \left(
    \begin{array}{ccc}
      \NORM{\vecBaseCCcv[1](\point)}^2 & 0 \\
      0 &\NORM{\vecBaseCCcv[2](\point)}^2\\
    \end{array}
  \right)
  =
  \left(
    \begin{array}{ccc}
      \metrTensCv{11}(\point)&     0       \\
      0             &\metrTensCv{22}(\point)\\
    \end{array}
  \right).
\end{equation}
The metric defines the surface scalar product
$\scalprodSurf{\uu}{\vv}=\metrTensCv{ij}\uu[i]\vv[j] $, and has
inverse denoted by $\First^{-1}=\{\metrTensCtrv{ij}\}$.
It is possible to show~\citep{book:Ciarlet2013} that there exist
constants $\EigFirst[*,\SurfDomain]$ and $\EigFirst[\SurfDomain]^*$
such that:
\begin{align*}
  \EigFirst[*,\SurfDomain]\NORM{\vv}^2
  &\le\scalprodSurf{\vv}{\vv}\le
    \EigFirst[\SurfDomain]^*\NORM{\vv}^2 \quad \text{for}\quad
    \vv\in\TanPlane{}{\SurfDomain}\,,
\end{align*}
where
$\EigFirst[*,\SurfDomain], \EigFirst[\SurfDomain]^*\ge 1$
are
the minimum and maximum eigenvalues of $\First$.  Moreover, we have
the following global uniform bounds on the norms and determinant of
$\First(\SurfDomain)$:
\begin{equation}\label{eq:metric-estimates}
  \EigFirstInvmin[\SurfDomain]=
  \inf_{\point\in\SurfDomain}\NORM{\First^{-1}(\point)}^2\,,
  \quad
  \EigFirstInvmax[\SurfDomain]=
  \sup_{\point\in\SurfDomain}\NORM{\First^{-1}(\point)}^2\,,
  \quad
  \CminDetG[\SurfDomain]\le\sqrt{\DET{\First(\SurfDomain)}}
  \le\CmaxDetG[\SurfDomain]\,.
\end{equation}
We will be using the symbol $\ABS{\SecondForm[\SurfDomain]}$ to denote
the supremum over $\SurfDomain$ of the norm of the second fudamental
form. This value can be related to the curvatures of $\SurfDomain$.

\paragraph{Differential operators}
Within this setting it is possible to define the relevant intrinsic
differential operators. For the intrinsic gradient of a scalar
function $\scalFun$ we have $ \GradSurf\scalFun =
\First[]^{-1}\Grad\scalFun$, where $\Grad\scalFun$ is the covariant
derivative of $\scalFun$ in the $\sv$ coordinates. We can write the
intrinsic divergence of a (contravariant) vector
$\vecFun=\vecFun[1]{}\vecBaseCCcv[1]+\vecFun[2]{}\vecBaseCCcv[2]$ as $
\DivSurf \vecFun = \Div \left( \sqrt{\DET{\First}} \vecFun
\right)/\sqrt{\DET{\First}}$.
Note that here the flux vector $\vecFun=-\DiffTens\GradSurf\scalFun$
in~\cref{eq:pde} is a vector tangent to $\SurfDomain$ for a general
(symmetric) diffusion tensor $\DiffTens$. Moreover, if
$\DiffTens=\DiffCoef\IDMat$, \cref{eq:pde} becomes the classical
Laplace-Beltrami operator, i.e.
$\DivSurf\left(\DiffTens\GradSurf\scalFun\right)=\DiffCoef\LapSurf\scalFun$,
where:
  \begin{equation*}
    \LapSurf\scalFun =
    \DivSurf\GradSurf\scalFun
    =\frac{1}{\sqrt{\metrTensCv{11}\metrTensCv{22}}}
    \left[
      \DerPar{}{\xcl}
      \left(
        \sqrt{\frac{\metrTensCv{22}}{\metrTensCv{11}}}\,
        \DerPar{\scalFun}{\xcl}
      \right)
      +
      \DerPar{}{\ycl}
      \left(
        \sqrt{\frac{\metrTensCv{11}}{\metrTensCv{22}}}\,
        \DerPar{\scalFun}{\ycl}
      \right)
    \right]
    \, .
  \end{equation*}
  The standard tools deriving from Stokes theorems hold with the
  intrinsic operators without any modification.  We first recall the
  formal definition of the integral of a function over a surface,
  which does not depend on the parametrization
  \citep{book:AbateTovena2012}:
\begin{definition}\label{def:integral}
  Let $\scalFun:\SurfDomain\to\REAL$ be a continuous
  function defined on a regular surface $\SurfDomain$ with
  parametrization given by $\MapU:\SubsetU\to\SurfDomain$.
  The integral of $\scalFun$ on $\SurfDomain$ is:
  \begin{equation*}
    \int_{\SurfDomain}\scalFun =
    \int_{\inverse{\MapU}(\SurfDomain)}(\scalFun\circ\MapU)
    \,\sqrt{\DET{\First}}\Diff\sv
    \; .
  \end{equation*}
\end{definition}
Then, the following intrinsic Green's formula holds:
\begin{equation}
  \label{eq:intrinsic-green-lemma}
  \int_{\SurfDomain}\scalprodSurf{\DiffTens\GradSurf\Sol}{\GradSurf\Test}=
  -\int_{\SurfDomain}\DivSurf\left(\DiffTens\GradSurf\Sol\right) \;\Test 
  + \int_{\SurfDomainBnd}
  \scalprodSurf{\DiffTens\GradSurf\Sol}{\conormal}\Test
  \;,
\end{equation}
where $\conormal\From\SurfDomain\To\REAL[2]$ denotes the vector
tangent to $\SurfDomain$ and normal to
$\SurfDomainBnd$ with components written with respect to the
curvilinear reference frame (i.e. $\conormal =
\mu^1\vecBaseCCcv[1] + \mu^2\vecBaseCCcv[2]$).

\subsection{Intrinsic variational formulation}

Without loss of generality we assume that $\SurfDomain$ is described
by a single (global) parametrization
$\MapU:\SubsetU\to\SurfDomain$. The ensuing results and the
definitions extend directly to the case of a surface defined by an
atlas assuming, as mentioned before, that the necessary transition
maps are smooth.

\paragraph{Function spaces}
We use standard definitions and notations for Sobolev
spaces~\citep{book:Ciarlet2013}, which can be directly extended to a
compact manifold $\SurfDomain$
(see~\citep{book:Taylor2010I,book:hebey2000}). We denote with
$\Lspace(\SurfDomain)$ and $\Sob{1}(\SurfDomain)$ the classical
Hilbert spaces on $\SurfDomain$. Explicitly:
\begin{align*}
  \Lspace(\SurfDomain)
  =\left\{\Test:\SurfDomain\to\REAL\;:\;
  \int_{\SurfDomain}\Test^2<\infty\right\}\,,\quad
  \Sob{1}(\SurfDomain)
  =\left\{\Test\in\Lspace(\SurfDomain) \;:\;
    \GradSurf\Test\in\left(\Lspace(\SurfDomain)\right)^2\right\}.
\end{align*}
Norms in $\Lspace(\SurfDomain)$ and $\Sob{1}(\SurfDomain)$ are denoted
with $\NORM[\Lspace(\SurfDomain)]{\cdot}$ and
$\NORM[\Sob{1}(\SurfDomain)]{\cdot}$, respectively, and are given by:
\begin{align*}
  \NORM[\Lspace(\SurfDomain)]{\Test}^2
  =\int_{\SurfDomain}\Test^2\quad\mbox{and}\quad
  \NORM[\Sob{1}(\SurfDomain)]{\Test}^2
  =\int_{\SurfDomain}\Test^2 + \int_{\SurfDomain}\ABSsurf{\GradSurf\Test}^2\,,
\end{align*}
where
$\ABSsurf{\GradSurf\Test}^2=\scalprodSurf{\GradSurf\Test}{\GradSurf\Test}$.
We will also use the $\Sob{2}$-seminorm given by:
\begin{equation*}
  \ABS[\Sob{2}(\SurfDomain)]{\Test}^2
  =\int_{\SurfDomain}\ABSsurf{\DerCov[2]\Test}^2\,,
\end{equation*}
where
$\ABSsurf{\DerCov[2]\Test}=\Trace{(\First^{-1}\DerCov[2]\Test)^2}$
with $\DerCov[2]\Test$ being the second covariant derivative of
$\Test$.
We note that:
\begin{equation*}
  \int_{\SurfDomain}\ABSsurf{\DerCov[2]\Test}^2
  \le {\EigFirstInvmax[\SurfDomain]}^2\left[
    \int_{\SurfDomain}\ABS{\Der[2]\Test}^2
    + C_{\MapU}^2
      \ABS{\SecondForm[\SurfDomain]}^2\int_{\SurfDomain}\ABS{\Der\Test}^2
  \right]\,,
\end{equation*}
where $C_{\MapU}$ is a generic constant depending on the parametrization,
from which, recalling that
$(\LapSurf\Test)^2=
[\Trace{\First^{-1}\DerCov[2]\Test}]^2
\ge\Trace{(\First^{-1}\DerCov[2]\Test)^2}$,
it is easy to prove that the following holds:
\begin{equation}
  \label{eq:h2seminorm-u-lap}
  c \NORM[\Lspace]{\Der[2]\Test}^2
  \le \ABS[\Sob{2}]{\Test}^2
 \le C \NORM[\Lspace]{\LapSurf\Test}^2\,.
\end{equation}
Note that the above inequalities are written without reference to the
domain as they can be set equally on the surface $\SurfDomain$ or the
chart $\MapU^{-1}(\SurfDomain)=\SubsetU$.

\paragraph{Intrinsic variational problem}
Now we can write the intrinsic variational formulation of the elliptic
equation~\eqref{eq:pde} in the LCS as:
\begin{problem}
  \label{pb:varformSurf}
  Find $\Sol\in\Sob{1}(\SurfDomain)$,
  with $\AvgSol=0$, such that:
  \begin{equation*}
    \BilinearStiff{\Sol}{\Test}
    =
    \Rhs{\Test} \quad \Forall\Test\in\Sob{1}(\SurfDomain)
    \, ,
  \end{equation*}
  where the ``stiffness'' bilinear form and the ``forcing'' linear
  form are given by:
  \begin{align*}
    \BilinearStiff{\Sol}{\Test}=
    \int_{\SurfDomain}
    \scalprodSurf{\DiffTens\GradSurf\Sol}{\GradSurf\Test}
\qquad\mbox{ and }\qquad %
    \Rhs{\Test}=\int_{\SurfDomain}\force\,\Test
    \, .
  \end{align*}
\end{problem}
Note that this is the direct extension to the surface $\SurfDomain$ of
the classical variational formulation of~\cref{eq:pde}. All the
integrals are still written in intrinsic coordinates of our LCS.

\paragraph{Poincar\'e inequality}
The dependence of the constants on the geometric characteristics of
the surface originates mainly from the use of Poincar\'e inequality,
which is given without proof in the following lemma.
\begin{lemma}\label{lemma:poincare}
  Let $\SurfDomain$ a $\Cont[1]$-regular surface without boundary and
  let $\Sol\in\Sob{1}(\SurfDomain)$ be a function with average given
  by $\AvgSol=\frac{1}{\ABS{\SurfDomain}}\int_{\SurfDomain}\Sol$. Then,
  there exists a constant $\CPoincare>0$ such that:
  \begin{equation*}
    \NORM[\Lspace(\SurfDomain)]{\Sol-\AvgSol}\le
    \CPoincare\NORM[\Lspace(\SurfDomain)]{\GradSurf\Sol}
    \qquad \Forall\;\Sol\in\Sob{1}(\SurfDomain)
  \end{equation*}
\end{lemma}
Poincar\'e inequality can be adapted to surfaces with Neumann or
Dirichlet boundaries in the usual way (see, e.g.,~\citep{art:Burman2018}).
From now on we assume for simplicity that $\AvgSol=0$.  As a
consequence, the $\Lspace$ the $\Sob{1}$ norms are equivalent, as
stated in the next straight-forward corollary.
\begin{corollary}\label{lemma:equivalent-norms}
  Let $\SurfDomain$ a $\Cont[1]$-regular surface and let
  $\Sol\in\Sob{1}(\SurfDomain)$ be a function with zero average,
  $\AvgSol=\frac{1}{\ABS{\SurfDomain}}\int_{\SurfDomain}\Sol=0$. Then,
  the following inequalities hold:
  \begin{equation*}
    \NORM[\Lspace(\SurfDomain)]{\GradSurf\Sol}\le
    \NORM[\Sob{1}(\SurfDomain)]{\Sol}\le
    \sqrt{1+\CPoincare^2}\NORM[\Lspace(\SurfDomain)]{\GradSurf\Sol}\,.
  \end{equation*}
\end{corollary}
Next we want to give detailed expressions of the constant $\CPoincare$
as a function of the geometric characteristics of $\SurfDomain$.  We
recall that the best Poincar\'e constant is related to the first
nonzero eigenvalue of the Laplace-Beltrami operator on
$\SurfDomain$. Indeed, it is easy to see that
$\CPoincare^2=\Eig^{-1}$.  Thus, we need to distinguish the three
cases of a compact surface with no boundary, a surface with Dirichlet
boundary $\SurfDomainBnd_D$ and a surface with Neumann boundary
$\SurfDomainBnd_N$. Typically, bounds on these eigenvalues are given
in terms of the Ricci curvature, but we note that for surfaces in
$\REAL^3$ the Ricci and the Gaussian curvature coincide up to a
positive multiplicative constant. For this reason we state everything
in terms of the latter. We can summarize these results in the
following lemma, whose proof can be found in \citep{book:Li2012}.
\begin{lemma}
  Let $\SurfDomain$ be a regular surface with Gaussian curvature
  $\GaussCurv$ bounded from below by a constant $-\Riccibound$
  ($\Riccibound>0$), and denote by $\diam(\SurfDomain)$ the longest
  geodesic distance between two points. Then, there exist two
  positive constants $C_1$ and $C_2$ such that:
    \begin{equation*}
     \label{eq:poincare-nobnd}
      \Eig^N\ge\frac{C_1}{\diam(\SurfDomain)^2}\ExpIL{-C_2
        \diam(\SurfDomain) \sqrt{\Riccibound}}\,,
      \quad\text{and}\quad
      \Eig^D\ge\frac{C_1}{\diam(\SurfDomain)^2}\ExpIL{-C_2
        (1+\diam(\SurfDomain) \sqrt{\Riccibound})}\,,
    \end{equation*}
    where $\Eig^N$ identifies the case of a compact surface without
    boundary or with Neumann boundary and $\Eig^D$ the case of
    Dirichlet boundary.
\end{lemma}
In what follows we will characterize all the constants in the ensuing
inequalities by explicitly keeping track of $\CPoincare$ to quantify
the geometrical effects of the properties of the surface domain on the
error estimates.

\paragraph{Well-posedness}
We list here the classical assumptions on the continuity and
coercivity of the bilinear form and continuity of the linear form
defining the weak formulation~\ref{pb:varformSurf}.
For $\BilinearStiff{\cdot}{\cdot}$ we need the obvious hypothesis
that the diffusion tensor $\DiffTens$ be
positive definite, i.e., there exist two positive constants
$\DiffEigMin$ and $\DiffEigMax$ such that:
\begin{equation}\label{eq:diffcoercive}
  \DiffEigMin\NORM[\First]{\vv}^2
  \le\scalprodSurf{\DiffTens\vv}{\vv}\le
  \DiffEigMax\NORM[\First]{\vv}^2 \quad\mbox{ for all }\quad
  \vv\in\TanPlane{}{\SurfDomain} \quad\mbox{ and for all }\quad
  \point\in\SurfDomain\,.
\end{equation}
Then we can state the following lemma.
\begin{lemma}
  \label{lem:coer-cont}
  The bilinear form $\BilinearStiff{\cdot}{\cdot}$
  in~\cref{pb:varformSurf} is coercive and continuous, i.e., for any
  $\Sol,\Test\in\Sob{1}(\SurfDomain)$ the following inequalities hold:
  \begin{equation}\label{eq:diffusionSPD}
    \BilinearStiff{\Sol}{\Sol}\ge
    \frac{\DiffEigMin}{1+\CPoincare^2}
    \NORM[\Sob{1}(\SurfDomain)]{\Sol}^2\,,
    \qquad
    \ABS{\BilinearStiff{\Sol}{\Test}}\le
    \DiffEigMax
    \NORM[\Sob{1}(\SurfDomain)]{\Sol}\NORM[\Sob{1}(\SurfDomain)]{\Test}\,.
  \end{equation}
  Moreover, the linear form $\Rhs{\cdot}$ is continuous:
  \begin{equation*}
    \Rhs{\Test}\le
    \NORM[\Lspace(\SurfDomain)]{\force}\NORM[\Lspace(\SurfDomain)]{\Test}\,.
  \end{equation*}
\end{lemma}
\begin{proof}
  The proof is a standard application of the equivalence between the
  $\Lspace$ and $\Sob{1}$ norms
  in~\cref{lemma:equivalent-norms} and the inequalities in
  \cref{eq:diffcoercive}:
  \begin{equation*}
    \BilinearStiff{\Sol}{\Sol}
    =
    \int_{\SurfDomain}\scalprodSurf{\DiffTens\GradSurf\Sol}{\GradSurf\Sol}
    \ge
    \DiffEigMin
    \NORM[\Lspace(\SurfDomain)]{\GradSurf\Sol}^2
    \ge
    \frac{\DiffEigMin}{1+\CPoincare^2}\NORM[\Sob{1}(\SurfDomain)]{\Sol}^2\,.
  \end{equation*}
  For the continuity we write:
  \begin{equation*}
    \ABS{\BilinearStiff{\Sol}{\Test}}
    \le
    \NORM[\Lspace(\SurfDomain)]{\DiffTens\GradSurf\Sol}
    \NORM[\Lspace(\SurfDomain)]{\GradSurf\Test}
    \le
    \DiffEigMax
    \NORM[\Sob{1}(\SurfDomain)]{\Sol}
    \NORM[\Sob{1}(\SurfDomain)]{\Test}\, .
  \end{equation*}
  The continuity of $\Rhs{\cdot}$ is simply an application of the
  Cauchy-Schwarz inequality.
\end{proof}
Under the above assumptions, the Lax-Milgram theorem holds:
\begin{lemma}[Lax-Milgram theorem]
  Let $\TestSpace$ be a Hilbert space and
  $\BilinearStiffSymbol:\TestSpace\times\TestSpace\To\REAL$ be a
  continuous and coercive bilinear form. For all continuous linear
  forms $\RhsSymbol:\TestSpace\To\REAL$ there exists a unique function
  $\Sol\in\TestSpace$ such that:
  \begin{equation*}
    \BilinearStiff{\Sol}{\Test}=\Rhs{\Test}
    \qquad \Forall\;\Test\in\TestSpace\,.
  \end{equation*}
\end{lemma}

\subsection{Intrinsic finite element method}
\label{sec:isfem}
 We recall that our guiding principle that justifies certain choices
 is to maintain the scheme as intrinsic as possible, i.e., we want to
 use only intrinsic geometric quantities and minimize the use of the
 surface embedding.
\paragraph{The surface triangulation}
Let
$\TriangH(\SurfDomain)=\cup_{i=1}^{\NCell}\Cell[i]=\Closure{\SurfDomain}$
be a geodesic surface triangulation of $\SurfDomain$, formed by
the union of non-intersecting surface triangles. 
We denote by $\SurfDomain[\meshparam]$ or
  $\TriangH(\SurfDomain[\meshparam])$ the
piecewise linear interpolation of $\SurfDomain$, i.e., the union of
2-simplices in $\REAL[3]$ having the same vertices of
$\TriangH(\SurfDomain)$ and characterized by the mesh parameter
$\meshparam$, the length of the longest chord between two triangle
vertices in $\TriangH(\SurfDomain)$.
We assume that $\TriangH(\SurfDomain[\meshparam])$ is shape-regular,
i.e., there exists a  constant $c>0$ independent of
$\meshparam$ such that $ \Inradius[\Cell]/\meshparam[\Cell]\ge c$
for all $\CellH\in\TriangH(\SurfDomain[\meshparam])$,
where $\Inradius[\Cell]$ is the radius of the circle inscribed in
$\CellH$ and $\meshparam[\Cell]$ is the longest side of $\CellH$.
By assumption, $\TriangH(\SurfDomain[\meshparam])$ is a closely
inscribed triangulation in the sense of~\citep{book:Morvan2008}, or
equivalently, in the sense of~\citep{art:Dziuk2013}, which means that
$\TriangH(\SurfDomain[\meshparam])\subset\Neigh[\delta]$, where
$\Neigh[\delta]$ is a tubular neighborhood of $\SurfDomain$ of radius
$\delta$ such that every point $\point\in\Neigh[\delta]$ has a unique
orthogonal projection onto $\SurfDomain$. As a consequence, for every
flat cell $\CellH\subset\SurfDomain[\meshparam]$ there corresponds a
unique curved cell $\Cell\subset\SurfDomain$, and this correspondence
is bijective. Note that, the assumption of $\TriangH(\SurfDomain)$
being geodesic is needed only to ensure the uniqueness of this
correspondence.
Finally, we introduce our working local chart
$\SubsetU[\meshparam]\subset\SubsetU$ and note that there exists an
affine transformation $\MapLin$ such that
$\CellH=\MapLin(\SubsetU[\meshparam])$.

\paragraph{Intrinsic spatial discretization}
Formally, we would like to work with the finite-dimensional
$\PONE$-conforming FEM space
$\TestSpaceApprox\subset\Sob{1}(\SurfDomain)$ given by:
\begin{equation}\label{eq:testspacediscr-1}
  \TestSpaceApprox=\{\TestApprox\in\Cont(\Triang(\SurfDomain))
  \text{ such that }
  \TestApprox|_{\scriptscriptstyle\Cell}\in\PONE(\Cell) \quad \Forall
  \;\Cell\in\Triang(\SurfDomain)\} \,.
\end{equation}
However, this definition is meaningless since we do not know how to
define $\PONE(\Cell)$, namely the space of first order polynomials in
$\Cell$. One way to circumvent this problem is proposed in
\citep{art:Sander2010,art:Sander2012}, where surface barycentric
coordinates are used to generalize the classical linear interpolation
on the surface. In this case the basis functions are nonlinear and
their calculation requires the solution of local cell-wise quadratic
minimization problems. Another approach is used in \citep{art:BMP21},
where the equations, the variational formulation, and the VEM bilinear
forms are defined on the chart. In this case, two-dimensional
Lagrangian linear basis functions can be defined in the usual way
directly on the chart and can be used to evaluate the needed
integrals. This approach requires complete knowledge of the surface
parametrization, since a triangulation of the chart is needed.

Maintaining the goal of ISFEM to make use of the surface embedding as
little as possible, the following strategy for the definition of the
basis functions can be devised.
Working on an element-by-element basis as typical of FE methods, for
each surface element $\Cell$ we consider the associated reference
element $\SubsetU[\meshparam]=\MapU[\meshparam]^{-1}(\Cell)$ 
eventually up to an affine transformation.
Here, we require $\MapU[\meshparam]$ to be a Monge
parametrization, i.e.,
$\MapU[\meshparam]=(\hgcscomp,\height(\hgcscomp))$, which always
exists locally because $\SurfDomain$ is regular.
On $\SubsetU[\meshparam]$ using the $\hxv$ coordinates,
we can define the standard local $\PONE$ basis functions
$\hnodalBasis{j}^{\scriptscriptstyle\Cell}(\hxv)$.
By means of the composition with the local parametrization
$\MapU[\meshparam]$, we can formally define our local (nonlinear)
basis function $\nodalBasis{j}^{\scriptscriptstyle\Cell}$ on $\Cell$
such that
$\nodalBasis{j}^{\scriptscriptstyle\Cell}(\xv):=\hnodalBasis{j}^{\scriptscriptstyle\Cell}
\circ\MapU[\meshparam]^{-1}(\xv)$.
Note that, on one hand we need to know the values of the basis
functions and their derivatives only at the quadrature points
$\rpoint_{i}\in \Cell$, on the other hand we want to use minimal information
about the parametrization of the surface.
In the case of $\PONE$ basis and trapezoidal quadrature rule, the
quadrature points are the triangle vertices on both $\Cell$ and
$\SubsetU[\meshparam]$, so that the basis function values are directly the nodal values
$\nodalBasis{j}^{\scriptscriptstyle\Cell}(\point[i])=\delta_{ij},\;
i,j=1,2,3$.
For the gradients, we need to perform a
transformation to represent them in the orthogonal local reference frame
$\{\vecBaseCCcv[1],\vecBaseCCcv[2]\}\in\TanPlane[\rpoint_i]{\Cell}$
at each quadrature point. 
Practical computation of the surface gradient at the point
$\rpoint_{i}$ yields:
\begin{equation}\label{eq:grad-with-W}
  \GradSurf\nodalBasis{j}^{\scriptscriptstyle\Cell} =
  \First[\rpoint_{i}]^{-1}\invJac[\rpoint_{i}]\hJac[\rpoint_{i}]
  \Grad\hnodalBasis{j}^{\scriptscriptstyle\Cell}  =\First[\rpoint_{i}]^{-1}\matW[\rpoint_{i}]\Grad\hnodalBasis{j}^{\scriptscriptstyle\Cell}\,.
\end{equation}
\begin{remark}
  In case the height function $\height$ is known, the Jacobian
  $\hJac[\rpoint_{i}]$ takes on the following
  form:
  \begin{equation*}
    \hJac[\rpoint_{i}]=[\hvecBaseCCcv[1](\rpoint_{i}),\hvecBaseCCcv[2](\rpoint_{i})]=
    \begin{pmatrix}
      1 & 0 \\
      0 & 1 \\
      \DerPar{\height}{\hxcg} & \DerPar{\height}{\hycg}
    \end{pmatrix}\,,
  \end{equation*}
  while if $\height$ is
  not known, we can recover $\hJac[\rpoint_i]$ from the input tangent
  vectors by projection. The Jacobian $\Jac[\rpoint_i]$ is given by
  $\Jac[\rpoint_i]=\left[ \vecBaseCCcv[1]
  (\rpoint_i), \vecBaseCCcv[2](\rpoint_i)\right]$, with $ \vecBaseCCcv[1]
  , \vecBaseCCcv[2]$ defined from $ \hvecBaseCCcv[1],
  \hvecBaseCCcv[2]$ by orthogonalization.
\end{remark}
\begin{remark}
  We would like to stress here that
  $\Grad\hat{f}$ is the gradient of $\hat{f}$ on the reference
  element calculated with respect to the
  $\hxv$ coordinates.
  On the other hand,  $\matW[\rpoint_{i}]\Grad\hat{f}$  is the same
  gradient evaluated with respect to the $\sv$ coordintes, which may
  not be necessarily orthogonal in $\SubsetU[\meshparam]$ but 
  correspond to orthogonal reference vectors in
  $\TanPlane[\rpoint_{i}]{\Cell}$. 
\end{remark}
\begin{remark}
 The ISFEM $\PONE$ basis functions described in~\cite{art:BFP21} are
 equivalent to the ISFEM basis functions defined above. Indeed, if we
 consider $\CellH$ as our reference element $\SubsetU[\meshparam]$,
 then we can write
 $\tilde{\nodalBasis{j}}^{\scriptscriptstyle\Cell}=\hnodalBasis{j}^{\scriptscriptstyle\Cell}\circ\MapLin$
 as a function in $\REAL[3]$.
 Moreover, ISFEM recovers the Surface FEM approach described
 in~\citep{art:Dziuk2013} if $\Cell=\CellH$.
  \end{remark}
The last step in the definition of our global basis functions
$\nodalBasis{k}$, $k=1,\ldots,\nNodes$, is to glue together as usual
the elemental components. Note that, because of tangent
  planes at vertices are defined uniquely, the resulting global basis
  functions are obviously conforming, albeit known only at the
  vertices.
In conclusion, every
function $\TestApprox$ in the functional space $\TestSpaceApprox$
can be written as:
\begin{equation}
  \label{eq:basis}
  \TestApprox =
  \Interpolant(\TestApprox)
  =\sum_{k=1}^{\nNodes}\coeffTest{k}\nodalBasis{k}\,,
\end{equation}
where $\Interpolant(\TestApprox)$ indicates the ISFEM interpolant of
$\TestApprox$, and $\coeffTest{k}$ are the nodal coefficients.
Hence, the intrinsic FEM variational formulation can be written in the
LCS as:
\begin{problem}
  \label{pb:intrinsic-discrete}
  Find $\SolApprox\in\TestSpaceApprox$ such that
  \begin{equation*}
    \BilinearStiff{\SolApprox}{\TestApprox}=
    \Rhs{\TestApprox} \quad \Forall\TestApprox\in\TestSpaceApprox
  \, ,
  \end{equation*}
  where the linear and bilinear forms are given by:
  \begin{align*}
    \BilinearStiff{\SolApprox}{\TestApprox}=
    \int_{\SurfDomain}
    \scalprodSurf{\DiffTens\GradSurf\SolApprox}{\GradSurf\TestApprox}
    \qquad   \mbox{ and } \qquad
  \Rhs{\TestApprox}=\int_{\SurfDomain}\forceApprox\,\TestApprox
  \, .
  \end{align*}
\end{problem}
\paragraph{Surface quadrature rules}

Up to the definition of the ISFEM space $\TestSpaceApprox$, no
numerical approximations are done until this point, since all the
operators and integrals are defined on $\Triang(\SurfDomain)$ whose
interior coincides with the surface $\SurfDomain$. We would like to
remain within this setting as much as possible. Approximation issues
arise when we need to practically compute quantities.  To this aim, we
assume that all the relevant geometric information related to the
surface are known in exact or approximate (but consistent) form at the
vertices of the triangulation and proceed by defining appropriate
quadrature rules.
In order to maintain optimal second order accuracy we need to provide
quadrature rules whose error is locally proportional to
$\meshparam[\Cell]^2$.  Thus we can consider surface extensions of the
trapezoidal and the mid-point quadrature rules for triangles, as
developed in~\citep{art:Georg1994}, as modified and effectively used
in~\citep{art:BP20,art:BFP21}. In this work we consider the
trapezoidal rule given by:
\begin{align}
  \label{eq:trap_rule}
  \int_{\Cell}\scalFun
  \approx \; \QuadRule[\meshparam]{\scalFun}=\frac{1}{3}\sum_{j=1}^3
  \scalFun(\point[j])\CellHArea     
  \,,
\end{align}
where $\CellHArea$ is the cell area and $\scalFun(\point[j])$ are the
evaluation of the function $\scalFun$ at the cell nodes. We note that
the above quadrature rule uses known information at the vertices of
$\Triang(\SurfDomain)$, and thus does not require interpolation as the
midpoint would.

\paragraph{Discrete norms}
We will be using the discrete grid norm $\NORM[\meshparam]{\scalFun}$
of a function $\scalFun\in\TestSpaceApprox$ defined as:
\begin{equation}\label{eq:discrete-norm}
  \NORM[\meshparam]{\scalFun}^2=
  \sum_{\Cell\in\TriangH(\SurfDomain)}\frac{\CellHArea}{3}
  \NORM[\meshparam,\Cell]{\scalFun}^2
  =
  \sum_{\Cell\in\TriangH(\SurfDomain)}\frac{\CellHArea}{3}
  \sum_{j=1}^{3}\scalFun(\point[j])^2
  =\NORM[\meshparam]{\scalFun_{\meshparam}}^2\, ,    
\end{equation}
where
$\scalFun_{\meshparam}=\{\scalFun_k\}_1^{\nNodes}=\{\scalFun(\point[k])\}_1^{\nNodes}$
is the vector of coefficients of the linear combination on the
basis of $\TestSpaceApprox$.
This norm is equivalent to the $\Lspace$-norm and, as a consequence,
to the $\Sob{1}$-norm. In fact, we can write:
\begin{equation*}
  \NORM[\Lspace(\SurfDomain)]{\scalFun}^2
  =\int_{\SurfDomain}\scalFun^2=
  \sum_{\Cell\in\TriangH(\SurfDomain)}
  \int_{\Cell}\left(\sum_{j=1}^{3}\scalFun(\point[j])\nodalBasis{j}^{\scriptscriptstyle\Cell}\right)^2
  =
  \scalprod{\scalFun_{\meshparam}}{\matrixMass\scalFun_{\meshparam}}\, ,
\end{equation*}
The last scalar product can be controlled on both sides by the
eigenvalues of the mass matrix $\matrixMass$ to yield:
\begin{equation}\label{eq:discrete-norm-prop}
  \frac{\CminDetG}{4}\NORM[\meshparam]{\scalFun_{\meshparam}}^2
  \le
  \NORM[\Lspace(\SurfDomain)]{\scalFun}^2
  \le
  \CmaxDetG\NORM[\meshparam]{\scalFun_{\meshparam}}^2 \,.
\end{equation}

\paragraph{ISFEM formulation}
Now all the ingredients of the ISFEM formulation are completed and we
can write:
\begin{problem}[ISFEM formulation]
  \label{pb:final-isfem}
  Find $\SolApprox\in\TestSpaceApprox$ such that
  \begin{equation*}
    \BilinearStiffApprox{\SolApprox}{\TestApprox}=
    \RhsApprox{\TestApprox} \quad \Forall\TestApprox\in\TestSpaceApprox
  \, ,
  \end{equation*}
  where the linear and bilinear forms are given by:
  \begin{align*}
    \BilinearStiffApprox{\SolApprox}{\TestApprox}=
    \sum_{\Cell\in\Triang(\SurfDomain)}
    \frac{\CellHArea}{3}\sum_{j=1}^3
    \scalprodSurf{\DiffTens(\point[j])\GradSurf\SolApprox(\point[j])}%
                 {\GradSurf\TestApprox(\point[j])}\,,
  \end{align*}
  and
  \begin{align*}
    \RhsApprox{\TestApprox}=
    \sum_{\Cell\in\Triang(\SurfDomain)}\frac{\CellHArea}{3}
    \sum_{j=1}^3\forceApprox(\point[j])\,\TestApprox(\point[j])
    \, .
  \end{align*}
\end{problem}

\section{Numerical analysis of ISFEM}

In this section we provide $\Sob{1}$ estimates showing that the ISFEM
achieves optimal convergence.
 The standard FEM theory is adapted to the intrinsic setting, with
 special attention to the analysis of the influence on convergence
 errors of surface geometric characteristics, such as, e.g., metric
 tensor and curvatures.  For this purpose we will introduce in our
 analysis different constants.
The symbol $C$ will denote a generic
constant not depending on $\meshparam[\Cell]$ nor on surface
properties. The symbol
$\ConstGeom[i]$ will be used to identify constants that
  are independent of $\meshparam[\Cell]$ but
  depend upon different surface geometric
quantities. When working on single elements, we will use the symbol $\ConstGeom[i,\Cell]$ to denote the
$i$-th constant defined on $\Cell$.

As usual in FEM theory, this effort will be divided in
  two parts. First, the local analysis in \cref{subsec:approx} will
  develop approximation and interpolation errors on a single
  triangle. Then, these local results will be combined in
  \cref{subsec:conv} to yield the final estimates on the full
  surface.

\subsection{Approximation errors on triangles}\label{subsec:approx}

We start by summarizing some known results
on approximation errors arising from the substitution of $\Cell$ with
$\CellH$~\cite{book:Morvan2008,art:Dziuk2013}. Then, interpolation and
quadrature error estimates are addressed.

  \subsubsection{Surface approximation errors}  
Given a point $\qpoint\in\CellH$, we denote by $\Prm(\qpoint)\in\Cell$
the orthogonal projection of $\qpoint$ onto $\Cell$ along the
direction $\normalSurf(\Prm(\qpoint))$, normal to the surface in
$\Prm(\qpoint)$.
We state here some results related to the approximation of surface
triangles, which can be easily extended to the entire surface. The
proofs can be found in \citep[lemma 4.1]{art:Dziuk2013}.
\begin{lemma}
  Given $\CellH$, $\Cell$ and the
  projection map $\Prm$, the following estimates hold:
  \begin{itemize}
    \item the distance between the approximate triangulation and the
      surface satisfies: 
      \begin{equation*}
        \max_{\qpoint\in\CellH}\ABS{\overrightarrow{\Prm(\qpoint)\qpoint}}
        \le C \meshparam[\Cell]^2 \,;
      \end{equation*}
    \item the ratio $\delta_{\meshparam}$ between the area measures
      $\Diff\sv$ and $\Diff\xv$ of the surface triangle $\Cell$ and
      its approximation $\CellH$, defined by
      $\Diff\sv=\delta_{\meshparam}\Diff\xv$, satisfies:
      \begin{equation*}
        \NORM[L^{\infty}]{1-\delta_{\meshparam}}
        \le C \meshparam[\Cell]^2 \,.
      \end{equation*}
    \end{itemize}
\end{lemma}
For any point $\qpoint\in\CellH$, we define the relative curvature of
$\CellH$ with respect to $\SurfDomain$ in $\qpoint$ as follows.
\begin{definition}
  Given $\SubsetW\subset\CellH$, the relative curvature
  $\omega_{\SurfDomain}(\qpoint)$ of any point $\qpoint\in\SubsetW$
  with respect to $\SurfDomain$ is
  \begin{equation*}
    \omega_{\SurfDomain}(\qpoint)=
    \ABS{\overrightarrow{\Prm(\qpoint)\qpoint}}
    \ABS{\SecondForm[\Prm(\qpoint)]}\;.
  \end{equation*}
  Then, the relative curvature of $\SubsetW$ is
  $\omega_{\SurfDomain}(\SubsetW)=
  \sup_{\qpoint\in\SubsetW}\omega_{\SurfDomain}(\qpoint)$.
\end{definition}

With reference to \citep{book:Morvan2008}, we can state the
following surface approximation results.
\begin{lemma}\label{prop:morvan}
  Given a geodesic triangulation $\TriangH(\SurfDomain)$ with surface
  triangles $\Cell$ and geodesic edges $\Edge$, and their
  approximations $\CellH$ and $\EdgeH$ in
  $\TriangH(\SurfDomain[\meshparam])$, the following results hold.
  \begin{enumerate}
\item \label{prop:morvan-1}
  The curvilinear length $\edgeLength$ of edge $\Edge$ is
  related to the Euclidean length $\edgeHLength$ of the chord
  $\EdgeH$ via the
  inequalities:
  \begin{equation*}
    \edgeHLength\le \edgeLength
    \le\frac{1}{1-\omega_{\SurfDomain}(\EdgeH)}\edgeHLength\;,
  \end{equation*}
  where $\omega_{\SurfDomain}(\EdgeH)$ is the relative curvature of
  $\EdgeH$ with respect to $\SurfDomain$.
\item \label{prop:morvan-2}
  The difference between the unit vector
  $\vv_{_{\overrightarrow{\point\qpoint}}}$ aligned to the chord
  $\EdgeH$ and the unit tangent vector $\vecBaseCCcv[\point]$ to the
  geodesic edge at $\point$ satisfies:
  \begin{equation*}
    \ABS{\vv_{_{\overrightarrow{\point\qpoint}}}-\vecBaseCCcv[\point]}
    \leq\frac{1}{2}\ABS{\SecondForm[\Cell]}\edgeLength \; .
  \end{equation*}
\item \label{prop:morvan-3}
  The surface area of the cell $\Cell$ is related to the planar
  area of $\CellH$ by the relation:
  \begin{equation*}
    \ABS{\CellArea- \CellHArea}\le
    C_{\Cell} \left(\DevAngle[\max]^2+\omega_{\SurfDomain}(\CellH)\right),
  \end{equation*}
  where $C_{\Cell}$ is a constant depending on $\Cell$ and $\DevAngle[\max]$ is the maximum over all points
  $\qpoint\in\CellH$ of the angle between the tangent planes
  $\TanPlane[\qpoint]{\CellH}$ and
  $\TanPlane[\Prm(\qpoint)]{\Cell}$.
\end{enumerate}
\end{lemma}
The following lemma is a straight-forward consequence of the above results:
\begin{lemma}\label{lem:relative-curvature}
  For any $\CellH\in\TriangH(\SurfDomain[\meshparam])$ we have:
  \begin{itemize}
  \item the relative curvature can be bounded by:
    \begin{equation*}
      \omega_{\SurfDomain}(\CellH)=\sup_{\qpoint\in\CellH}
        \ABS{\overrightarrow{\Prm(\qpoint)\qpoint}}\ABS{\SecondForm[\Prm(\qpoint)]}\leq
        C\ABS{\SecondForm[\Cell]}\meshparam[\Cell]^2\;;
    \end{equation*}
  \item the maximum angle between tangent planes of $\Cell$ and
    $\CellH$ can be bounded by:
    \begin{equation*}
    \DevAngle[\max]
    \leq C \ABS{\SecondForm[\Cell]}\meshparam[\Cell]
 \; .
  \end{equation*}
  \end{itemize}
\end{lemma}

\subsubsection{Interpolation and quadrature errors}

We start with estimates of the interpolation errors in
$\TestSpaceApprox$. We first note that that, using the metric bounds
in~\cref{eq:metric-estimates}, it is easy to prove the following
inequalities relating $\Lspace$-norms of a function $\scalFun$ and its
gradient $\Grad\scalFun$ in the cell $\Cell$ and in the chart
$\SubsetU[\meshparam]$:
\begin{align*}
  \CminDetG[\Cell]\NORM[\Lspace({\SubsetU[\meshparam]})]{\scalFun}^2
  \le\NORM[\Lspace({\Cell})]{\scalFun}^2
    &=\int_{\Cell}\scalFun^2
    =\int_{\SubsetU[\meshparam]}\scalFun^2\sqrt{\DET{\First}}\Diff\sv
    \le \CmaxDetG[\Cell]\NORM[\Lspace({\SubsetU[\meshparam]})]{\scalFun}^2\,,\\
  \begin{split}
    \CminDetG[\Cell]\EigFirstInvmin[\Cell]
    \NORM[\Lspace({\SubsetU[\meshparam]})]{\Grad\scalFun}^2
    \le\NORM[\Lspace({\Cell})]{\GradSurf\scalFun}^2
    &=\int_{\Cell}\ABS{\GradSurf\scalFun}^2\\
    &=\int_{\SubsetU[\meshparam]}
    \ABS{\First^{-1}\Grad\scalFun}^2\sqrt{\DET{\First}}\Diff\sv
    \le \CmaxDetG[\Cell]\EigFirstInvmax[\Cell]
    \NORM[\Lspace({\SubsetU[\meshparam]})]{\Grad\scalFun}^2\,.
  \end{split}
\end{align*}
The following lemma provides the interpolation error estimate for
$\Interpolant(\scalFun)$ defined in~\cref{eq:basis}.
\begin{lemma}[Interpolation error]\label{lemma:interp}
  Given a function $\scalFun\in\Sob{1}(\Cell)$, let
  $\Interpolant(\scalFun)$ be the ISFEM interpolant in
  \cref{eq:basis}. Then, we have:
  \begin{align*}
    \NORM[\Lspace(\Cell)]{\scalFun-\Interpolant(\scalFun)}
    &\le C\ConstGeom[1,\Cell]\meshparam[\Cell]^2
      \NORM[\Lspace({\SubsetU[\meshparam]})]{\Der[2]\scalFun}\,,\\[1em]
    \NORM[\Lspace(\Cell)]{\GradSurf\scalFun-\GradSurf
    \Interpolant(\scalFun)}
    &\le C\ConstGeom[2,\Cell]\meshparam[\Cell]
      \NORM[\Lspace({\SubsetU[\meshparam]})]{\Der[2]\scalFun}\,,
  \end{align*}
  where $\ConstGeom[1,\Cell]=\sqrt{\CmaxDetG[\Cell]}$ and 
    $\ConstGeom[2,\Cell]=\sqrt{\EigFirstInvmax[\Cell]\CmaxDetG[\Cell]}$.
\end{lemma}
\begin{proof}
  Since $\TriangH(\SurfDomain)$ is assumed to be shape-regular, using
  the standard planar interpolation error and definition of
  $\Interpolant(\scalFun)$, we can write:
  \begin{align*}
    \NORM[\Lspace(\Cell)]{\scalFun-\Interpolant(\scalFun)}^2
    \le \CmaxDetG[\Cell] 
      \NORM[\Lspace({\SubsetU[\meshparam]})]
         {\scalFun\circ\MapU[\meshparam]-\ProjFun[1]{\scalFun}}^2
     \le   C^2 \,\CmaxDetG[\Cell]\, \meshparam[\Cell]^4
      \NORM[\Lspace({\SubsetU[\meshparam]})]{\Der[2]\scalFun}^2
      \,,
  \end{align*}
  where $\ProjFun[1]{\scalFun}$ is the standard $\PONE$ interpolation
  of $\scalFun$ in $\SubsetU[\meshparam]$.
  For the gradient we obtain:
  \begin{multline*}
    \NORM[\Lspace(\Cell)]
         {\GradSurf\scalFun-\GradSurf\Interpolant(\scalFun)}^2 \le
         \EigFirstInvmax[\Cell] \,\CmaxDetG[\Cell]
         \NORM[\Lspace({\SubsetU[\meshparam]})]
              {\Grad(\scalFun\circ\MapU[\meshparam])-\Grad\ProjFun[1]{\scalFun}}^2\\
               \le C^2 \EigFirstInvmax[\Cell]\, \CmaxDetG[\Cell]\,
              \meshparam[\Cell]^2
              \NORM[\Lspace({\SubsetU[\meshparam]})]
                   {\Der[2]\scalFun}^2\,,
  \end{multline*}
  where the gradient of the interpolant
  is intended in our $LCS$. 
\end{proof}

Next, we switch our attention to the accuracy of the surface
quadrature rule. Following the results in~\cite{art:BP20} we show that
the surface trapezoidal rule converges with optimal quadratic rate.
\begin{lemma}[{Surface Trapezoidal rule}]
  \label{lemma:trap}
  Given a function $\scalFun\From\Cell\To\REAL$, the
    surface trapezoidal rule is given by:
  \begin{equation*}
    \QuadRule[\meshparam,T]{\scalFun}=
    \frac{\CellHArea}{3}\sum_{j=1}^{3}\scalFun(\point[j])
  \end{equation*}
  and satisfies:
  \begin{equation*}
    \ABS{\int_{\Cell}\scalFun - \QuadRule[\meshparam,T]{\scalFun}}
    \le
    C \meshparam[\Cell]^2
    \left(\NORM[\Lspace({\SubsetU[\meshparam]})]{\Der[2]\scalFun}
    \ConstGeom[1,\Cell] + \NORM[\meshparam,\Cell]{\scalFun}\ConstGeom[3,\Cell]\right)\, ,
  \end{equation*}
  where $\scalFun(\point[i])$ is the value of
  $\scalFun$ at the triangle
  vertices and
  $\ConstGeom[3,\Cell]=
  C_{\Cell}\ABS{\SecondForm[\Cell]}(\ABS{\SecondForm[\Cell]}+1)$.
\end{lemma}
\begin{proof}
  We denote by  $\QuadRule{\scalFun}$ the surface
  integral of the projection of $\scalFun$ onto $\TestSpaceCell$,
  i.e., $\QuadRule{\scalFun}=\int_{\Cell}\Interpolant(\scalFun)$.
  Application of the triangular inequality to the quadrature error
  yields:
  \begin{align*}
    \ABS{\int_{\Cell}\scalFun-\QuadRule[\meshparam,T]{\scalFun}}
    &\le \ABS{\int_{\Cell}\scalFun-\QuadRule{\scalFun}}
      +\ABS{\QuadRule{\scalFun}-\QuadRule[\meshparam,T]{\scalFun}}\,.
  \end{align*}
  From~\cref{lemma:interp} the first term can be bounded by:
  \begin{align*}
    \ABS{\int_{\Cell}\scalFun-\QuadRule{\scalFun}}
    \le
    C \meshparam[\Cell]^2
    \NORM[\Lspace({\SubsetU[\meshparam]})]{\Der[2]\scalFun}
    \ConstGeom[1,\Cell]\, ,
  \end{align*}
  while for the second term we
  use~\cref{prop:morvan},~\cref{prop:morvan-3}:
  \begin{align*}
    \ABS{\QuadRule{\scalFun}-\QuadRule[\meshparam,T]{\scalFun}}
    &\le \frac{1}{3}\ABS{\sum_{j=1}^3\scalFun(\point[j])}
      \ABS{\CellArea-\CellHArea}
   \le \frac{1}{3}\left(\sum_{j=1}^3\scalFun(\point[j])^2\right)^{1/2}
      \ABS{\CellArea-\CellHArea}\\
    &
    \le \frac{1}{3}\NORM[\meshparam,\Cell]{\scalFun}\,C_{\Cell}
    \left(\DevAngle[\max]^2+\omega_{\SurfDomain}(\CellH)\right) 
    \le \frac{1}{3}\NORM[\meshparam,\Cell]{\scalFun}\,C_{\Cell}
    \ABS{\SecondForm[\Cell]}(\ABS{\SecondForm[\Cell]}+1)\, .
  \end{align*}
  Putting the two inequalities together we obtain:
  \begin{equation*}
    \ABS{\int_{\Cell}\scalFun-\QuadRule[\meshparam,T]{\scalFun}}
    \le
    C \meshparam[\Cell]^2
    \left(
      \NORM[\Lspace({\SubsetU[\meshparam]})]{\Der[2]\scalFun}
      \ConstGeom[1,\Cell] + \NORM[\meshparam,\Cell]{\scalFun}\,C_{\Cell}
      \ABS{\SecondForm[\Cell]}(\ABS{\SecondForm[\Cell]}+1)
    \right)\, .
  \end{equation*}

\end{proof}
As a remark, we note that,
analogously, the midpoint rule is characterized by
a similar error estimate, given by:
\begin{equation*}
  \ABS{\int_{\Cell} \scalFun - \QuadRule[\meshparam,M]{\scalFun}}
  \le
  C \meshparam[\Cell]^2
  \left(\NORM[\Lspace({\SubsetU[\meshparam]})]{\Der[2]\scalFun}
    \ConstGeom[1,\Cell] + \NORM[\infty,\Cell]{\scalFun}\ConstGeom[3,\Cell]\right)\, ,
\end{equation*}
where $\scalFun(\midPointCell)$ is the value of $\scalFun$ at the
centroid $\midPointCell$ of $\Cell$, having nodal coordinates
given by $\sv(\midPointCell)=\sum_{i=1}^3\sv(\point[i])/3$.

\subsection{Convergence analysis}\label{subsec:conv}

  In this section we collect the previously developed local error
  estimates to build the global estimates forming the overall
  convergence theory of ISFEM. This analysis proceeds following a
  standard FEM approach by combining consistency with interpolation
  errors.
  Obviously, the developments must take into account the fact that our
  linear and bilinear ISFEM forms are approximated using the
  trapezoidal or the midpoint rule. This is handled in a usual fashion
  with the help of a discrete mesh norm $\NORM[\meshparam]{\cdot}$,
  which is shown to be equivalent to the $\Lspace$-norm on
  $\SurfDomain$ to show coercivity of the discrete bilinear form. Then
  the surface extension of Strang lemma paves the way for the proofs
  of consistency and then of the final theorems on convergence in
  $\Sob{1}$ and $\Lspace$ norms. The latter, being exactly the surface
  extension of the Nitsche-Aubin duality trick, is only mentioned
  without proof.
  Everything will be done with the implicit assumption that the
  surface trapezoidal rule is employed. The results for the midpoint
  rule are exactly the same, and can be proved in a similar way.

  We start our task by showing that the discrete bilinear form is
  coercive.
\begin{lemma} \label{lemma:discrete-coerc}
  The discrete bilinear form $\BilinearStiffApprox{\cdot}{\cdot}$ in
  \cref{pb:final-isfem} satisfies:
  \begin{equation*}
    \BilinearStiffApprox{\TestApprox}{\TestApprox}
    \ge
    \frac{\DiffEigMin\EigFirst[*,\SurfDomain]}{\CmaxDetG[\SurfDomain](1+\CPoincare^2)}
    \NORM[\Sob{1}(\SurfDomain)]{\TestApprox}^2
  \end{equation*}
\end{lemma}
\begin{proof}
  Using Poincar\'e inequalities (\cref{lemma:equivalent-norms}) and
  the equivalence between the discrete and $\Lspace$-norms
  \cref{eq:discrete-norm-prop}, we obtain immediately:
  \begin{align*}
    \BilinearStiffApprox{\TestApprox}{\TestApprox}
    &=\sum_{\Cell\in\Triang(\SurfDomain)}
      \frac{\CellHArea}{3}\sum_{j=1}^3
      \scalprodSurf{\DiffTens(\point[j])\GradSurf\TestApprox(\point[j])}{\GradSurf\TestApprox(\point[j])}\\
    &\ge  \DiffEigMin\EigFirst[*,\SurfDomain]
      \NORM[\meshparam]{\GradSurf\TestApprox}^2
      \ge \frac{\DiffEigMin\EigFirst[*,\SurfDomain]}{\CmaxDetG[\SurfDomain](1+\CPoincare^2)}
      \NORM[\Sob{1}(\SurfDomain)]{\TestApprox}^2\,.
  \end{align*}

\end{proof}

The next step is the surface version of Strang lemma
(see~\citep[Theorem 4.1.1]{book:ciarlet2002finite} for the classical
version of the proof):
\begin{lemma}[Surface Strang-like lemma]
  Let $\Sol\in\Sob{1}(\SurfDomain)$ and
  $\SolApprox\in\TestSpaceApprox$ be solutions of~\cref{pb:varformSurf}
  and~\cref{pb:final-isfem}, respectively. Then, the error
  $\Sol-\SolApprox$ satisfies:
  \begin{multline*}
    \NORM[\Sob{1}(\SurfDomain)]{\Sol-\SolApprox}
    \le\inf_{\TestApprox}\left[
      \left(1+\DiffEigMax
        \frac{\CmaxDetG[\SurfDomain](1+\CPoincare^2)}%
        {\DiffEigMin\EigFirst[*,\SurfDomain]}
      \right)
      \NORM[\Sob{1}(\SurfDomain)]{\Sol-\TestApprox}\right.\\
    \left.
      +\frac{\CmaxDetG[\SurfDomain](1+\CPoincare^2)}%
      {\DiffEigMin\EigFirst[*,\SurfDomain]}
      \sup_{\TestWApprox}
      \frac{\ABS{\BilinearStiff{\TestApprox}{\TestWApprox}
              -\BilinearStiffApprox{\TestApprox}{\TestWApprox}}}%
           {\NORM[\Sob{1}(\SurfDomain)]{\TestWApprox}}
         \right]\\
         + \frac{\CmaxDetG[\SurfDomain](1+\CPoincare^2)}%
         {\DiffEigMin\EigFirst[*,\SurfDomain]}
         \sup_{\TestWApprox}
         \frac{\ABS{\Rhs{\TestWApprox}-\RhsApprox{\TestWApprox}}}%
         {{\NORM[\Sob{1}(\SurfDomain)]{\TestWApprox}}}\,.
  \end{multline*}
\end{lemma}
\begin{proof}
  From the continuity of $\BilinearStiff{\cdot}{\cdot}$
  (\cref{lem:coer-cont}) and the triangle inequality, we
  can write for all $\TestApprox,\TestWApprox\in\TestSpaceApprox$:
  \begin{multline*}
    \ABS{\BilinearStiffApprox{\SolApprox-\TestApprox}{\TestWApprox}}
    =\ABS{\BilinearStiff{\Sol-\TestApprox}{\TestWApprox}
      +\BilinearStiff{\TestApprox}{\TestWApprox}
      -\BilinearStiffApprox{\TestApprox}{\TestWApprox}
      -\Rhs{\TestWApprox}+\RhsApprox{\TestWApprox}}\\
    \le\DiffEigMax
      \NORM[\Sob{1}(\SurfDomain)]{\Sol-\TestApprox}\NORM[\Sob{1}(\SurfDomain)]{\TestWApprox}
      +\ABS{\BilinearStiff{\TestApprox}{\TestWApprox}
      -\BilinearStiffApprox{\TestApprox}{\TestWApprox}}
      + \ABS{\Rhs{\TestWApprox}-\RhsApprox{\TestWApprox}}\,.
  \end{multline*}
  Using the coercivity of the discrete bilinear form, we obtain:
  \begin{equation*}
    \BilinearStiffApprox{\SolApprox-\TestApprox}{\SolApprox-\TestApprox}
    \ge
    \frac{\DiffEigMin\EigFirst[*,\SurfDomain]}{\CmaxDetG[\SurfDomain](1+\CPoincare^2)}
    \NORM[\Sob{1}(\SurfDomain)]{\SolApprox-\TestApprox}
    \inf_{\TestWApprox}\NORM[\Sob{1}(\SurfDomain)]{\TestWApprox}\,,
  \end{equation*}
  or, equivalently:
  \begin{equation*}
    \NORM[\Sob{1}(\SurfDomain)]{\SolApprox-\TestApprox}
    \le \frac{\CmaxDetG[\SurfDomain](1+\CPoincare^2)}{\DiffEigMin\EigFirst[*,\SurfDomain]}
    \sup_{\TestWApprox}
    \frac{
      \ABS{\BilinearStiffApprox{\SolApprox-\TestApprox}{\TestWApprox}}
    }%
    {\NORM[\Sob{1}(\SurfDomain)]{\TestWApprox}}\,.
  \end{equation*}
  Putting together the two inequalities we obtain:
  \begin{align*}
    \NORM[\Sob{1}(\SurfDomain)]{\Sol-\SolApprox}
    &\le
      \NORM[\Sob{1}(\SurfDomain)]{\Sol-\TestApprox}
      +\NORM[\Sob{1}(\SurfDomain)]{\SolApprox-\TestApprox}\\
    &\le
      \NORM[\Sob{1}(\SurfDomain)]{\Sol-\TestApprox} +
      \frac{\CmaxDetG[\SurfDomain](1+\CPoincare^2)}{\DiffEigMin\EigFirst[*,\SurfDomain]}
      \sup_{\TestWApprox}
      \frac{
      \ABS{\BilinearStiffApprox{\SolApprox-\TestApprox}{\TestWApprox}}
      }%
      {\NORM[\Sob{1}(\SurfDomain)]{\TestWApprox}}\,,
  \end{align*}
  and taking the infimum over $\TestApprox\in\TestSpaceApprox$, the
  result follows.
\end{proof}

Note that the combination of the estimates of the interpolation error
(\cref{lemma:interp}) and the trapezoidal rule error
(\cref{lemma:trap}) shows that the scheme converges with optimal
(first) order of accuracy in the $\Sob{1}(\SurfDomain)$-norm.
Indeed, it is possible to connect the constants of the error estimates
with the geometric characteristics of $\SurfDomain$, as the following
consistency lemma states.

\begin{lemma}[Consistency]\label{lemma:bilinear-consistency}
  For any continuous and coercive bilinear functional
  $\BilinearStiffSymbol:\TestSpace\times\TestSpace\To\REAL$ and any
  continuous linear functional $\RhsSymbol:\TestSpace\To\REAL$ as
  given in \cref{pb:varformSurf}, the discrete approximations
  $\BilinearStiffSymbol_{\meshparam}:
  \TestSpaceApprox\times\TestSpaceApprox\To\REAL$ and
  $\RhsSymbol_{\meshparam}:\TestSpaceApprox\To\REAL$ given in
  \cref{pb:final-isfem} are consistent. In other words, we have that
  for, any $\TestWApprox\in\TestSpaceApprox$,:
  \begin{enumerate}[a)]
  \item for the bilinear form $\BilinearStiff{\cdot}{\cdot}$:
    \begin{multline}\label{eq:consistency-bilinear}
      \ABS{\BilinearStiff{\TestApprox}{\TestWApprox}-
        \BilinearStiffApprox{\TestApprox}{\TestWApprox}} \le\\
       C \meshparam^2s
    \left(\frac{\ConstGeom[1]}{{\CminDetG[\SurfDomain]\EigFirstInvmin[\SurfDomain]}}\NORM[{\infty,\SurfDomain}]{\Der[2]\left(\matW^T\DiffTens\,\First^{-1}\matW\right)} + \ConstGeom[3]\frac{2\DiffEigMax}{\sqrt{\CminDetG[\SurfDomain]}} \right)\NORM[\Sob{1}(\SurfDomain)]{\TestApprox}\NORM[\Sob{1}(\SurfDomain)]{\TestWApprox}\,,
    \end{multline}
    where 
    $\NORM[{\infty,\SurfDomain}]{\Der[2]\left(\matW^T\DiffTens\,\First^{-1}\matW\right)}$
    is the maximum over the element of the sup-norm of
    $\Der[2]\left(\matW^T\DiffTens\,\First^{-1}\matW\right)$, and
    $\ConstGeom[i]=\max_{\Cell}\ConstGeom[i,\Cell]$, for $i=1,3$.

  \item for the linear form $\RhsApprox{\cdot}$:
    \begin{equation}\label{eq:consistency-force}
      \ABS{\Rhs{\TestWApprox}-\RhsApprox{\TestWApprox}}    
      \le
      C \meshparam^2
    \left(\frac{\ConstGeom[1]}{\sqrt{\CminDetG[\SurfDomain]}}\NORM[{\Lspace(\SurfDomain)}]{\Der[2]\force} + \ConstGeom[3]\NORM[\meshparam]{\force}\right)\NORM[\Sob{1}(\SurfDomain)]{\TestWApprox}\,.
    \end{equation}
  \end{enumerate}
\end{lemma}
  
\begin{proof}
  The proof is an application of~\cref{lemma:trap} to the special
  cases of integrands in the linear and bilinear forms.
  For the first part of the lemma \cref{eq:consistency-bilinear}, we
  start by writing:
  \begin{multline}\label{eq:strang-bound-stiff}
    \ABS{\BilinearStiff{\TestApprox}{\TestWApprox}-
      \BilinearStiffApprox{\TestApprox}{\TestWApprox}}\\
    =
    \ABS{
      \int_{\SurfDomain}
      \scalprodSurf{\DiffTens\GradSurf\TestApprox}{\GradSurf\TestWApprox}
      -
      \sum_{\Cell\in\Triang(\SurfDomain)}
      \frac{\CellHArea}{3}\sum_{j=1}^3
      \scalprodSurf{\DiffTens(\point[j])\GradSurf\TestApprox(\point[j])}
                   {\GradSurf\TestWApprox(\point[j])}}\\
                   \le
                   \sum_{\Cell\in\Triang(\SurfDomain)}
                   \ABS{  \int_{\Cell}\Gfun - 
                     \frac{\CellHArea}{3}\sum_{j=1}^3\Gfun(\point[j])}\\
                   \le
                   \sum_{\Cell\in\Triang(\SurfDomain)} C \meshparam[\Cell]^2
                   \left(\NORM[\Lspace({\SubsetU[\meshparam]})]{\Der[2]\Gfun}
                     \ConstGeom[1,\Cell]
                     + \NORM[\meshparam,\Cell]{\Gfun}\ConstGeom[3,\Cell]\right)\,
  \end{multline}
  where
  $\Gfun=\scalprodSurf{\DiffTens\GradSurf\TestApprox}{\GradSurf\TestWApprox}$.
  To estimate $\NORM[\Lspace({\SubsetU[\meshparam]})]{\Der[2]\Gfun}$,
  we write $\Gfun$ on the reference element $\SubsetU[\meshparam]$ in
  the $\hxv$ coordinate using~\cref{eq:grad-with-W} obtaining
  $\scalprodSurf{\DiffTens\,\First^{-1}\matW\Grad\bTestApprox}{\First^{-1}\matW\Grad\bTestWApprox}=\scalprod{\matW^T\DiffTens\,\First^{-1}\matW\Grad\bTestApprox}{\Grad\bTestWApprox}$. Then,
  using standard two-dimensional arguments, we obtain the following
  inequality:
  \begin{multline*}
       \NORM[{\Lspace(\SubsetU[\meshparam])}]{\Der[2]\Gfun}
      \le
        \NORM[{\infty,\Cell}]{\Der[2]\left(\matW^T\DiffTens\,\First^{-1}\matW\right)}
      \NORM[{\Lspace(\SubsetU[\meshparam])}]{\Grad\bTestApprox}
      \NORM[{\Lspace(\SubsetU[\meshparam])}]{\Grad\bTestWApprox}\\
      \le
      \frac{1}{\CminDetG[\Cell]\EigFirstInvmin[\Cell]}
      \NORM[{\infty,\Cell}]{\Der[2]\left(\matW^T\DiffTens\,\First^{-1}\matW\right)}
      \NORM[{\Lspace(\Cell)}]{\Grad\TestApprox}
      \NORM[{\Lspace(\Cell)}]{\Grad\TestWApprox}
      \,,
  \end{multline*}
  %
%
  where we have used the fact that $\Der[2]\bTestApprox=0$ for
  $\PONE$ test functions. The final estimate is obtained by summation,
  where $\NORM[{\infty,\SurfDomain}]{\Der[2]\left(\matW^T\DiffTens\,\First^{-1}\matW\right)}$
  is the maximum over all the elements of the sup-norm of
  $\Der[2]\left(\matW^T\DiffTens\,\First^{-1}\matW\right)$.
  Recalling that
  $\NORM[\meshparam]{\Gfun}=\NORM[\meshparam]{\Interpolant(\Gfun)}$,
  from continuity of the bilinear form and \cref{eq:discrete-norm-prop},
  we obtain the following estimate directly written on $\SurfDomain$:
  \begin{equation*}
    \NORM[\meshparam]{\Gfun}^2
    \le
    \frac{4}{\CminDetG[\SurfDomain]}{\DiffEigMax}^2  
    \NORM[\Sob{1}(\SurfDomain)]{\TestApprox}^2\NORM[\Sob{1}(\SurfDomain)]{\TestWApprox}^2
    \, .
  \end{equation*}    
Finally, we can put the previous inequalities together and sum over all
the elements.
  The consistency estimate for the quadrature-based bilinear
  form then becomes: 
  \begin{multline*}
    \ABS{\BilinearStiff{\TestApprox}{\TestWApprox}-
      \BilinearStiffApprox{\TestApprox}{\TestWApprox}}\\
    \le C \meshparam^2
    \left(
      \frac{\ConstGeom[1]}{{\CminDetG[\SurfDomain]\EigFirstInvmin[\SurfDomain]}}\NORM[{\infty,\SurfDomain}]{\Der[2]\left(\matW^T\DiffTens\,\First^{-1}\matW\right)} + \ConstGeom[3]\frac{2\DiffEigMax}{\sqrt{\CminDetG[\SurfDomain]}} \right)\NORM[\Sob{1}(\SurfDomain)]{\TestApprox}\NORM[\Sob{1}(\SurfDomain)]{\TestWApprox}\,,
  \end{multline*}
  where we set $\ConstGeom[i]=\max_{\Cell}\ConstGeom[i,\Cell]$, for $i=1,3$.

  The proof of \cref{eq:consistency-force} proceeds analogously with
  the definition of 
  $\Gfun$ replaced by $\force\TestWApprox$.
\end{proof}

The previous lemma together with the interpolation and
quadrature errors yield the main convergence theorem for ISFEM.
\begin{theorem}[Optimal $\Sob{1}$-norm convergence]
  The ISFEM approach converges in the $\Sob{1}$-norm with optimal
  first order accuracy:
  \begin{equation*}
    \NORM[\Sob{1}(\SurfDomain)]{\Sol-\SolApprox}
    \le \mathbb{C}_1 \meshparam \NORM[\Lspace(\SurfDomain)]{\force}
    + \mathbb{C}_2 \meshparam^2 \NORM[\Sob{2}(\SurfDomain)]{\force}
  \end{equation*}
  where the constants $\mathbb{C}_1,\mathbb{C}_2$ depend on the
  surface $\SurfDomain$ and its geometrical characteristics and on
  the upper and lower bounds of the diffusion tensor $\DiffTens$, but
  not upon the surface discretization.
\end{theorem}
\begin{proof}
  The proof is a direct application of the previous two lemmas.
  Indeed, including the consistency estimates into the
  surface Strang-like lemma we obtain:
  \begin{align*}
    &\NORM[\Sob{1}(\SurfDomain)]{\Sol-\SolApprox}
    \le\inf_{\TestApprox}\left[
      \left(1+\DiffEigMax\frac{\CmaxDetG[\SurfDomain](1+\CPoincare^2)}{\DiffEigMin\EigFirst[*,\SurfDomain]}\right)
      \NORM[\Sob{1}(\SurfDomain)]{\Sol-\TestApprox}\right.\\
    &\quad\quad\left.
      +\frac{\CmaxDetG[\SurfDomain](1+\CPoincare^2)}{\DiffEigMin\EigFirst[*,\SurfDomain]}
      \sup_{\TestWApprox}
      \frac{\ABS{\BilinearStiff{\TestApprox}{\TestWApprox}
              -\BilinearStiffApprox{\TestApprox}{\TestWApprox}}}%
           {\NORM[\Sob{1}(\SurfDomain)]{\TestWApprox}}
         \right]
           + \frac{\CmaxDetG[\SurfDomain](1+\CPoincare^2)}{\DiffEigMin\EigFirst[*,\SurfDomain]}
\sup_{\TestWApprox}
      \frac{\ABS{\Rhs{\TestWApprox}-\RhsApprox{\TestWApprox}}}%
           {{\NORM[\Sob{1}(\SurfDomain)]{\TestWApprox}}}\\
    &\le
      C\left(1+\DiffEigMax\frac{\CmaxDetG[\SurfDomain](1+\CPoincare^2)}{\DiffEigMin\EigFirst[*,\SurfDomain]}\right)
      \frac{\left(\ConstGeom[2]+\ConstGeom[1]\meshparam\right)}{\sqrt{\CminDetG[\SurfDomain]}}
      \meshparam
      \NORM[\Lspace(\SurfDomain)]{\Der[2]\Sol}
\\
    &\quad\quad
      +C \frac{\CmaxDetG[\SurfDomain](1+\CPoincare^2)}{\DiffEigMin\EigFirst[*,\SurfDomain]}
      \meshparam^2
      \Bigg[\ConstGeom[3]\left(\frac{2\DiffEigMax}{\sqrt{\CminDetG[\SurfDomain]}}\NORM[\Sob{1}(\SurfDomain)]{\Interpolant(\Sol)}+\NORM[\meshparam]{\force}\right)\\
    &\qquad\qquad
      +\ConstGeom[1]
      \left( \frac{1}{\CminDetG[\SurfDomain]\EigFirstInvmin[\SurfDomain]}\NORM[{\infty,\SurfDomain}]{\Der[2]\left(\matW^T\DiffTens\,\First^{-1}\matW\right)}\NORM[\Sob{1}(\SurfDomain)]{\Interpolant(\Sol)}
    +\frac{\NORM[\Lspace(\SurfDomain)]{\Der[2]\force}}{\sqrt{\CminDetG[\SurfDomain]}}\right)\Bigg]
      \, , 
  \end{align*}
  where we chose $\TestApprox=\Interpolant(\Sol)$ and we used the
  inequality:
  \begin{multline*}
    \NORM[\Sob{1}(\SurfDomain)]{\TestApprox}^2
      \le
      \frac{\CmaxDetG[\SurfDomain](1+\CPoincare^2)}{\DiffEigMin\EigFirst[*,\SurfDomain]} \BilinearStiffApprox{\TestApprox}{\TestApprox}\\
      =
      \frac{\CmaxDetG[\SurfDomain](1+\CPoincare^2)}{\DiffEigMin\EigFirst[*,\SurfDomain]}\RhsApprox{\TestApprox}
    \le
      \frac{\CmaxDetG[\SurfDomain](1+\CPoincare^2)}{\DiffEigMin\EigFirst[*,\SurfDomain]}
      \frac{4}{{\CminDetG[\SurfDomain]}}
      \NORM[\Lspace(\SurfDomain)]{\forceApprox}\NORM[\Lspace(\SurfDomain)]{\TestApprox}\,.
  \end{multline*}
Note that the last inequality comes directly from the continuity of $\RhsApprox{\TestApprox}$.
\end{proof}

  The estimate in the above theorem contains on the right-hand-side
  a standard term related to the interpolation error plus a second
  term that goes to zero for a flat surface and it is thus
    compatible with standard FEM $\PONE$ estimates.
Finally, we conclude this section by mentioning only that
using the standard duality arguments optimal $\Lspace$
convergence is obtained.

\section{Numerical experiments}
\label{sec:results}

In this section we provide numerical support to the results presented
in the previous sections.
We present two different numerical experiments. In the first test case
the aim is to show that the value of the constant in the
$\Lspace$-error for the solution and its gradient increases while the
maximum value of the curvature increases. The second test case shows
that the ISFEM method can be directly applied in the presence of
multiple charts, if compatible sets of tangent vectors are
available. Both test cases were considered already in
\citep{art:BMP21} for the case of advection-diffusion-reaction
equation discretized on the chart and numerically solved by means of a
geometrically intrinsic version of the virtual element method. For the
test case 2, we extend the result in \citep{art:BMP21} by directly
solving the equation on the sphere, without the use of extra
conditions at the interface of the two charts. 

In all the experiments we consider a manufactured solution
$\Sol\From\SurfDomain\To\REAL$ and calculate the resulting forcing
function $\force$ by substitution into the original equation.  Note
that, even a simple manufactured solution would become highly
nonlinear when considered on a surface, due to the spatially varying
geometric information.

\subsection*{Test case 1}
\begin{figure}
  \centerline{
    \includegraphics[width=0.31\textwidth]{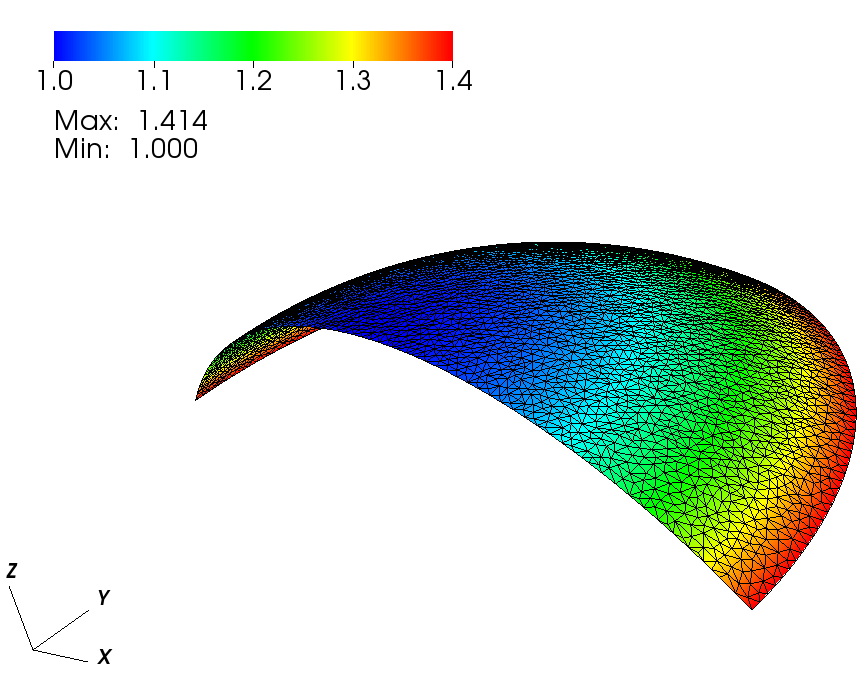}
    \includegraphics[width=0.31\textwidth]{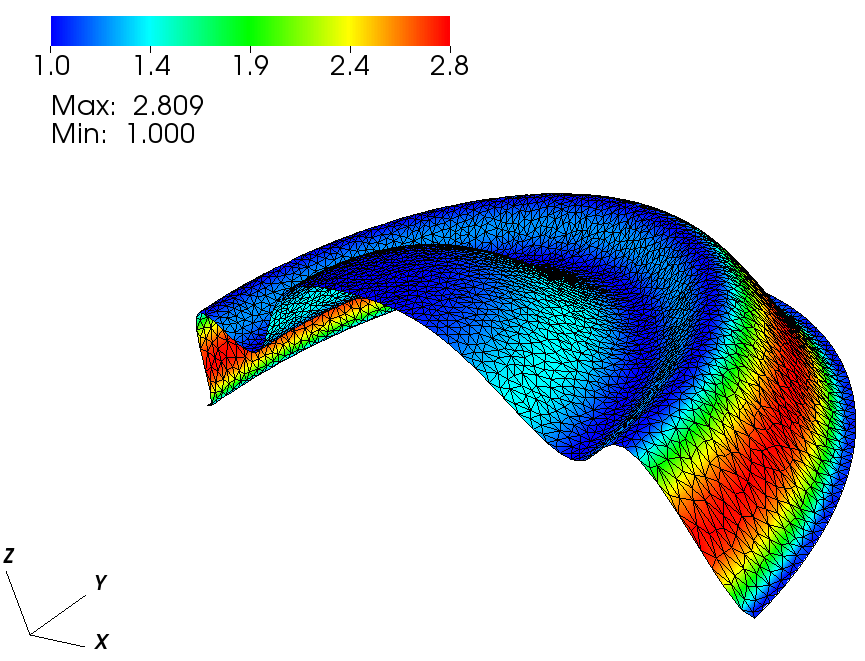}
    \includegraphics[width=0.31\textwidth]{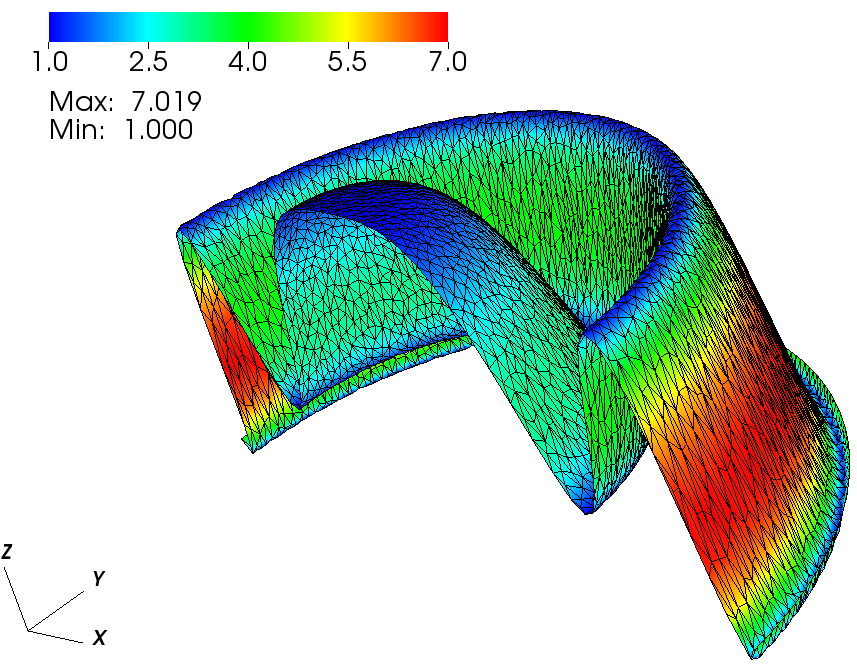}
  }
  \caption{Surfaces described by \cref{eq:surface-trig} with $r=2$,
    $k=5$ and $a=0, 0.5, 2$, respectively in the left, middle and
    right panels. The color map shows the value of
    $\sqrt{\DET{\First(\SurfDomain)}}$.  }
  \label{fig:surfaces}
\end{figure}

For the test case 1 we consider the surface provided by the graph of
the following height function (see \citep{art:BMP21}):
\begin{equation}\label{eq:surface-trig}
  \zcg=\height(\xcg,\ycg)=\sqrt{r - (\xcg)^2 - (\ycg)^2 +
    a \cos^2\left(k\frac{\pi}{2} ((\xcg)^2 + (\ycg)^2)\right)} \; ,
\end{equation}
a trigonometric perturbation of a sphere, where $r$ is the radius of
the sphere, and $a$ and $k$ are the amplitude and the frequency of the
cosine trigonometric perturbation.
We use the Monge parametrization and define the surface by
$\SurfDomain=\{(\xcg,\ycg,\height(\xcg,\ycg)) \;|\;
\ycg\ge0 \mbox{ and } (\xcg)^2+(\ycg)^2\le 1\}$, with a radius $r=2$ and a
frequency $k=5$.
Figure~\ref{fig:surfaces} shows the surfaces obtained with different
values of the amplitude $a$: a sphere ($a=0$) is shown in the left
panel, and two trigonometric deformations of the sphere are shown in
the middle and left panels for the case $a=0.5$ and $a=2$,
respectively. The color map shows the distribution in space of
$\sqrt{\DET{\First(\SurfDomain)}}$.
The mesh sets used in this test case are obtained from subsequent
refinements of Delaunay triangulations of
$\SubsetU = \{(\xcg,\ycg) \;|\; \ycg\ge0 \mbox{ and }
(\xcg)^2+(\ycg)^2\le 1\}$, then elevated using the height function
\cref{eq:surface-trig}. We consider 8 levels of refinement, with a
initial value of the surface mesh parameter $\meshparam \approx 0.25$
at $\Lev=0$. This corresponds to a total of 70 surface nodes for the
case of $a=0$ and $a=0.05$, while for the case $a=2$ the total number
of nodes is 265 for $\Lev=0$.
We define $\Sol=\xcg$ as manufactured solution and we compute an
expression for the forcing function by the equation
$\force(\xcl,\ycl)=-\LapSurf\Sol$. We apply Dirichlet boundary
condition by imposing the exact solution at the boundary nodes.

\begin{figure}
  \includegraphics[width=0.95\linewidth]{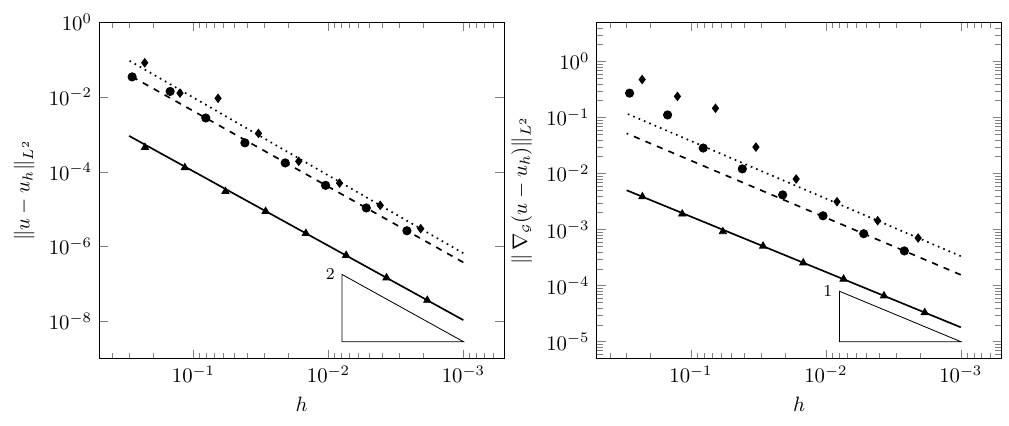}
  \caption{TC1: Numerical convergence of $\Lspace$-errors for the
    solution (left) and its gradient (right) vs $\meshparam$ on the
    surface triangulation. The convergence lines are obtained by means
    of least-square approximation considering the last 2 point
    values. The different lines denote the three different values of
    $a$ considered: solid line with triangular data points is used for
    the case $a=0$, dashed line with circular data points for $a=0.5$,
    and dotted line with diamond data points for the case $a=2$. The
    optimal theoretical slope is represented by the lower right
    triangles.}
  \label{fig:tc1:conv}
\end{figure}

\Cref{fig:tc1:conv} shows the $\Lspace$-errors and experimental
orders of convergence for the solution and its gradient. We notice
second order convergence rates for the solution and first order for
the gradient. Convergence rates slightly different than the optimal
ones at the initial levels are attributable to a too coarse
resolution of the surface triangulation, not accurate enough to
approximate the surface. In particular, this phenomenon can be observed
in the case of higher values of the parameter $a$ that corresponds to
higher values of the surface curvature.

\subsection*{Test case 2}
\begin{figure}
  \centerline{
    \includegraphics[width=0.45\linewidth]{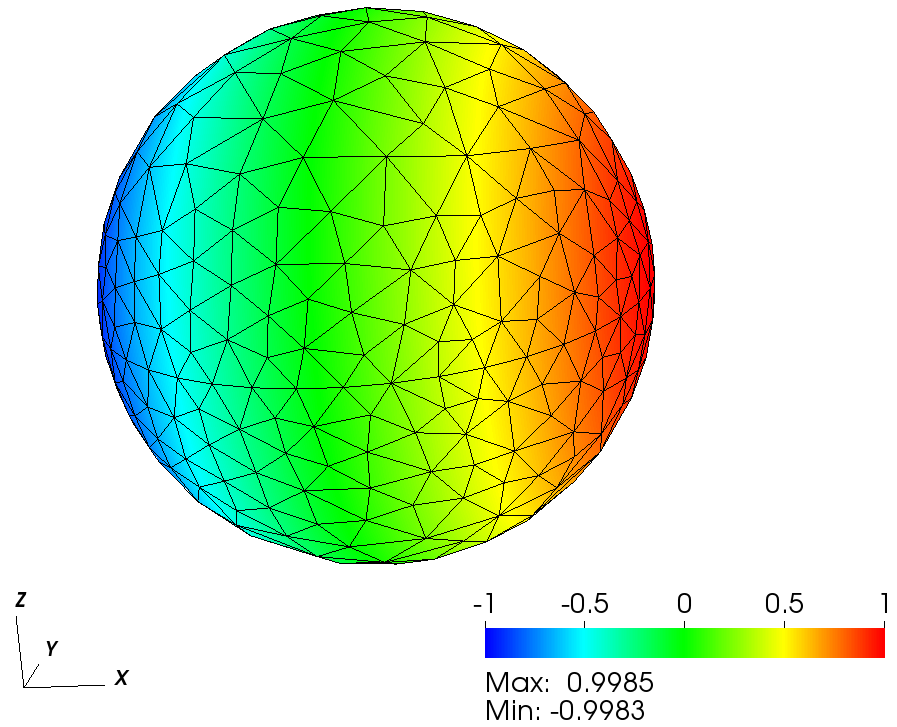}
    \includegraphics[width=0.45\linewidth]{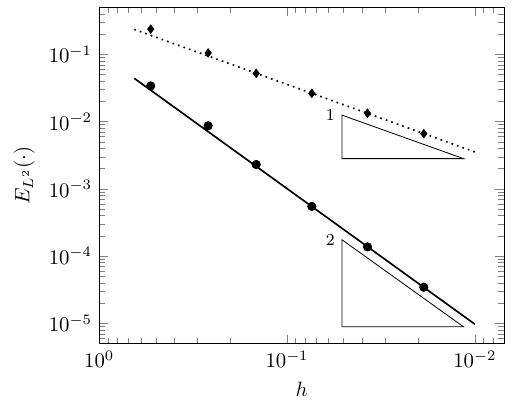}
  }
  \caption{TC2: Numerical solution on the sphere (mesh level
    $\Lev=1$), left panel, and numerical convergences in the $L2-$norm
    for both the solution (solid line with circular data points) and
    its gradient (dotted line with diamond data points), right panel.
    The convergence lines are obtained by approximating via
    least-square the last 3 point values.}
  \label{fig:sphere}
\end{figure}

We consider here $\SurfDomain=S^2$.  We use two parametrizations, one
for the northern and one for the southern hemispheres, to define two
sets of tangent vectors and consider the equator as intersection
set. A smooth transition map is known in this particular case. The
stereographic projections for the northern hemisphere is given by:
\begin{equation}
  \label{eq:projN}
  \MapU[N](\xcl,\ycl)
  =\left(\frac{2\xcl}{1+(\xcl)^2+(\ycl)^2},
    \frac{2\ycl}{1+(\xcl)^2+(\ycl)^2},
    \frac{1-(\xcl)^2-(\ycl)^2}{1+(\xcl)^2+(\ycl)^2}\right)
  =(\xcg,\ycg,\zcg)\,,
\end{equation}
and similarly for the southern hemisphere, with the transition map
between north and south written as:
\begin{equation*}
  \psi(\xcl[N],\ycl[N])=\left(\frac{\xcl[N]}{(\xcl[N])^2+(\ycl[N])^2},
      \frac{\ycl[N]}{(\xcl[N])^2+(\ycl[N])^2}\right) = (\xcl[S],\ycl[S])\,.
\end{equation*}
The surface triangulation of the sphere is obtained by computing a
Delaunay triangulation of the disk,
$\SubsetU = \{(\xcl,\ycl) \;|\; (\xcl)^2+(\ycl)^2\le 1\}$, then
projecting the points to the surface using the two stereographic
projections. We consider a set
of meshes with 6 levels of refinement, with a initial value of the
surface mesh parameter $\meshparam=0.532$ and a total of 111 surface
nodes.
Analogously to the first test case, we assume $\Sol=\xcg$ and compute
the forcing function by $\force(\xcl,\ycl)=-\LapSurf\Sol+\Sol$, where
$\xcg\ne\xcl$ because of the use of the stereographic projection. We
note that, contrary to test case 1, we use the inverse of
\cref{eq:projN} (as well as the south projection) to compute
$(\xcl,\ycl)$ from $(\gcscomp)$ in the evaluation of the forcing
function.
\Cref{fig:sphere} show on the left the manufactured solution and on
the right the $\Lspace$-errors and experimental orders of
convergence for the solution and its gradient. Again, we notice an
optimal order of convergence rates for both the solution and the
gradient.

\section*{Acknowledgments}
This study received funding from the European Union - Next Generation
EU National Recovery and Resilience Plan [NRRP], Mission 4, Component
2, Investment 1.3–D.D. 1243 2/082022, PE0000005 Extended Partnership
“RETURN: Multi- Risk sciEnce for resilienT commUnities underR a
changiNg climate” and Mission 4, Component 2, Investment 1.5 - Call
for tender No. 3277 of 30 dicembre 2021 ECS00000043, No. 1058 of June
23, 2023, CUP C43C22000340006, “iNEST: Interconnected Nord-Est
Innovation Ecosystem”. The authors are members of the Gruppo Nazionale
Calcolo Scientifico - Istituto Nazionale di Alta Matematica
(GNCS-INdAM).

\bibliography{strings,biblio}
\bibliographystyle{abbrvnat}

\end{document}

%% file: newcommands.tex
\usepackage{mathrsfs}
\usepackage{scalerel}

\DeclareMathAlphabet{\Mymathbb}{U}{bbold}{m}{n}
\DeclareMathAlphabet{\mathpzc}{OT1}{pzc}{m}{it}

 \newtheorem{theorem}{Theorem}[subsection]
 
 \newtheorem{corollary}[theorem]{Corollary}
 \newtheorem{definition}[theorem]{Definition}
 \newtheorem{lemma}[theorem]{Lemma}
 
 \newtheorem{remark}[theorem]{Remark}
 \newtheorem{problem}[theorem]{Problem}

\newcommand{\Forall}{\forall \,}

\newcommand{\FirstSymbol}{\mathcal{G}}
\newcommand{\First}[1][]{
  \ifthenelse{\equal{#1}{}}
  {\FirstSymbol}
  {\FirstSymbol_{\scriptscriptstyle{#1}}}
}
\newcommand{\EigFirst}[1][]{
  \ifthenelse{\equal{#1}{}}
  {\mu}
  {\mu_{\scriptscriptstyle{#1}}}
}
\newcommand{\EigFirstInvmin}[1][]{
  \ifthenelse{\equal{#1}{}}     
  {c_{*}}
  {c_{*,\scriptscriptstyle{#1}}}
}
\newcommand{\EigFirstInvmax}[1][]{
  \ifthenelse{\equal{#1}{}}     
  {{c^{*}}}
  {{c^{*}_{\scriptscriptstyle{#1}}}}
}
\newcommand{\CminDetG}[1][]{
  \ifthenelse{\equal{#1}{}}     
  {g_{*}}
  {g_{*,\scriptscriptstyle{#1}}}
}  
\newcommand{\CmaxDetG}[1][]{
  \ifthenelse{\equal{#1}{}}     
  {g^*}
  {g^*_{\scriptscriptstyle{#1}}}
}  
\newcommand{\ConstGeom}[1][]{
   \ifthenelse{\equal{#1}{}}       
   {K\,}
   {K_{{#1}}\,}
}

\newcommand{\FirstForm}[1][]{
  \ifthenelse{\equal{#1}{}}       
  {\operatorname{I_{\point}}}             
  {\operatorname{I_{#1}}}             
}
\newcommand{\GradSymbol}{\operatorname{\mathbf{\nabla}}}
\newcommand{\Grad}{\GradSymbol}

\newcommand{\Div}{\Grad\cdot}
\newcommand{\Lap}{\Delta}
\newcommand{\GradSurf}{\GradSymbol_{\!\scriptscriptstyle\First}}
\newcommand{\DivSurf}{\GradSurf\cdot}

\newcommand{\LapSurf}{\Lap_{\scriptscriptstyle\First}}

\newcommand{\Der}[1][]
{
  \ifthenelse{\equal{#1}{}}
  {\partial}
  {\partial^{\scriptscriptstyle{#1}}}
}
\newcommand{\DerCov}[1][]
{
  \ifthenelse{\equal{#1}{}}
  {\nabla}
  {\nabla^{{#1}}}
}

\newcommand{\DerPar}[2]{\frac{\Der #1}{\Der #2}}

\newcommand{\DerTot}[2][t]
{
  \ifthenelse{\equal{#2}{}}
  {\frac{d #2}{dt}}
  {\frac{d #2}{d #1}}
}
\newcommand{\Diffsymbol}{\operatorname{d}\!}
\newcommand{\Diff}[2][]{
  \ifthenelse{\equal{#1}{}}
  {\Diffsymbol{#2}}
  {\Diffsymbol{#2}_{{#1}}}
}
\newcommand{\DirDerSymb}{D}
\newcommand{\DirDer}[2][]
{
  \ifthenelse{\equal{#1}{}}
  {\DirDerSymb^{#2}}
  {\DirDerSymb^{#2}_{#1}}
}

\newcommand{\ExpIL}[1]{\ensuremath{\operatorname{exp}\left({#1}\right)}}
\newcommand{\REALsymbol}{\mathbb{R}}
\newcommand{\REAL}[1][]{
  \ifthenelse{\equal{#1}{}}
  {\REALsymbol}
  {{\REALsymbol}^{#1}}
}
\newcommand{\EUCLsymbol}{\mathbb E}
\newcommand{\EUCL}[1][]{
  \ifthenelse{\equal{#1}{}}
  {\EUCLsymbol}
  {{\EUCLsymbol}^{#1}}
}
\newcommand{\NATURALsymbol}{\mathbb N}
\newcommand{\NATURAL}[1][]{
  \ifthenelse{\equal{#1}{}}
  {\NATURALsymbol}
  {{\NATURALsymbol}^{#1}}
}

\newcommand{\ABS}[2][]
{
  \ifthenelse{\equal{#1}{}}
  {\left| #2 \right|}
  {\left| #2 \right|_{#1}}
}

\newcommand{\ABSsurf}[1]{\ensuremath{\left|#1\right|_{\scriptscriptstyle\First}}}
\newcommand{\NORM}[2][]
{
  \ifthenelse{\equal{#1}{}}
  {\left\| #2 \right\|}
  {\left\| #2 \right\|_{#1}}
}
\newcommand{\vertiii}[1]{{\left\vert\kern-0.25ex\left\vert\kern-0.25ex\left\vert #1 \right\vert\kern-0.25ex\right\vert\kern-0.25ex\right\vert}}
\newcommand{\BrNORM}[2][]
{
  \ifthenelse{\equal{#1}{}}
  {\vertiii{#2}}
  {\vertiii{#2}_{#1}}
}

\newcommand{\SCAL}[3][]
{
  \ifthenelse{\equal{#1}{}}
  {\left\langle{#2},{#3}\right\rangle}
  {\left\langle{#2},{#3}\right\rangle_{#1}}
}
\newcommand{\SCALF}[3][]
{
  \ifthenelse{\equal{#1}{}}
  {\left({#2},{#3}\right)}
  {\left({#2},{#3}\right)_{#1}}
}

\newcommand{\scalprod}[2]{\left\langle #1,#2 \right\rangle}
\newcommand{\scalprodSurf}[3][]
{
  \ifthenelse{\equal{#1}{}}
  {\left\langle {#2},{#3} \right\rangle_{\scriptscriptstyle\First}}
  {\left\langle {#2},{#3} \right\rangle_{{#1}}}
}

\newcommand{\DET}[1]{\ensuremath{\operatorname{det}(#1)}}
\newcommand{\Trace}[1]{\ensuremath{\operatorname{tr}(#1)}}

\newcommand{\point}[1][]
{
  \ifthenelse{\equal{#1}{}}
  {\mathbf{p}}
  {\mathbf{p}_{#1}}
}

\newcommand{\qpoint}{\mathbf{q}}

\newcommand{\rpoint}{\mathbf{r}}

\newcommand{\midPoint}{\mathbf{m}}

\newcommand{\From}{:}
\newcommand{\To}{\longrightarrow}
\newcommand{\Mapsto}{\longmapsto}
\newcommand{\Closure}[1]{\operatorname{cl}({#1})}

\newcommand{\RegionSymb}{R}
\newcommand{\Region}[1][]
{
  \ifthenelse{\equal{#1}{}}
  {\RegionSymb}
  {\RegionSymb_{#1}}
}

\newcommand{\diam}{\rho}

\newcommand{\Interval}[1][]
{
  \ifthenelse{\equal{#1}{}}
  {I}
  {I_{#1}}
}

\newcommand{\SurfDomainsymb}{\Gamma}
\newcommand{\SurfDomain}[1][]{
  \ifthenelse{\equal{#1}{}}
  {\SurfDomainsymb}
  {\SurfDomainsymb_{#1}}
}
\newcommand{\tSurfDomain}[1][]{
  \ifthenelse{\equal{#1}{}}
  {\tilde{\SurfDomainsymb}}
  {\tilde{\SurfDomainsymb}_{#1}}
}
\newcommand{\SurfDomainBndsymb}{\partial\Gamma}
\newcommand{\SurfDomainBnd}[1][]{
  \ifthenelse{\equal{#1}{}}
  {\SurfDomainBndsymb}
  {\SurfDomainBndsymb_{#1}}
}
\newcommand{\ClosedSurfDomain}[1][]{
  \ifthenelse{\equal{#1}{}}
  {\Closedsymb{\SurfDomain}}
  {\Closedsymb{\SurfDomain[#1]}}
}

\newcommand{\heightsymb}{\mathcal{H}}
\newcommand{\height}[1][]{
  \ifthenelse{\equal{#1}{}}
  {\heightsymb}
  {\heightsymb_{#1}}
}

\newcommand{\BSMsymbol}{\mathcal{B}}

\newcommand{\BSM}[1][]
{
  \ifthenelse{\equal{#1}{}}
  {\BSMsymbol}
  {\BSMsymbol_{#1}}
}

\newcommand{\Surf}{\mathcal{S}}

\newcommand{\SurfBndsymb}{\partial\Surf}
\newcommand{\SurfBnd}[1][]{
  \ifthenelse{\equal{#1}{}}
  {\SurfBndsymb}
  {\SurfBndsymb_{#1}}
}
\newcommand{\Closedsymb}[1]{\bar{#1}}
\newcommand{\ClosedSurf}[1][]{
  \ifthenelse{\equal{#1}{}}
  {\Closedsymb{\Surf}}
  {\Closedsymb{\Surf[#1]}}
}

\newcommand{\Vector}[1]{\mathbf{#1}}

\newcommand{\MapUsymb}{\phi}
\newcommand{\MapU}[1][]
{
  \ifthenelse{\equal{#1}{}}
    {\MapUsymb}
    {\MapUsymb_{\scriptscriptstyle{#1}}}
}
\newcommand{\MapLinsymb}{\mathcal{F}}
\newcommand{\MapLin}[1][]
{
  \ifthenelse{\equal{#1}{}}
  {\ensuremath{\MapLinsymb}}
  {\ensuremath{\MapLinsymb_{\scriptscriptstyle{#1}}}}
}

\newcommand{\MapVsymb}{\psi}
\newcommand{\MapV}[1][]
{
  \ifthenelse{\equal{#1}{}}
    {\MapVsymb}
    {\MapVsymb_{\scriptscriptstyle{#1}}}
}
\newcommand{\Transsymb}{\Phi}
\newcommand{\Trans}[1][]
{
  \ifthenelse{\equal{#1}{}}
    {\Transsymb} 
    {\Transsymb_{\scriptscriptstyle{#1}}}
}
\newcommand{\InvMapsymb}{\Psi}
\newcommand{\InvMap}[1][]
{
  \ifthenelse{\equal{#1}{}}
    {\InvMapsymb}
    {\InvMapsymb_{\scriptscriptstyle{#1}}}
}

\newcommand{\fsymb}{f}
\newcommand{\scalFun}[1][]
{
  \ifthenelse{\equal{#1}{}}
  {\fsymb}
  {\fsymb_{#1}}
}
\newcommand{\tscalFun}[1][]
{
  \ifthenelse{\equal{#1}{}}
  {\tilde{\fsymb}}
  {\tilde{\fsymb}_{#1}}
}
\newcommand{\bscalFun}[1][]
{
  \ifthenelse{\equal{#1}{}}
  {\bar{\fsymb}}
  {\bar{\fsymb}_{#1}}
}
\newcommand{\gsymb}{g}
\newcommand{\scalFung}[1][]
{
  \ifthenelse{\equal{#1}{}}
  {\gsymb}
  {\gsymb_{#1}}
}
\newcommand{\Fsymb}{F}
\newcommand{\FvecFun}[1][]
{
  \ifthenelse{\equal{#1}{}}
  {\Fsymb}
  {\Fsymb_{#1}}
}
\newcommand{\tFvecFun}[1][]
{
  \ifthenelse{\equal{#1}{}}
  {\tilde{\Fsymb}}
  {\tilde{\Fsymb}_{#1}}
}
\newcommand{\Ffunc}[2][]
{
  \ifthenelse{\equal{#1}{}}
  {\ensuremath{\Fsymb_{#2}}}
  {\ensuremath{\Fsymb^{#1}_{#2}}}
}
\newcommand{\Gsymb}{g}
\newcommand{\Gfun}[1][]
{
  \ifthenelse{\equal{#1}{}}
  {\ensuremath{\Gsymb}}
  {\ensuremath{\Gsymb_{{#1}}}}
}
\newcommand{\bGfun}[1][]
{
  \ifthenelse{\equal{#1}{}}
  {\ensuremath{\hat{\Gsymb}}}
  {\ensuremath{\hat{\Gsymb}_{{#1}}}}
}

\newcommand{\inverse}[1]{#1^{-1}}

\newcommand{\JacSymb}{\mathbf{J}}
\newcommand{\Jac}[1][]
{
  \ifthenelse{\equal{#1}{}}
  {\JacSymb}
  {{\JacSymb}_{\scriptscriptstyle{#1}}}
}
\newcommand{\invJac}[1][]
{
  \ifthenelse{\equal{#1}{}}
  {\JacSymb^{\scriptscriptstyle{+}}}
  {{\JacSymb}^{\scriptscriptstyle{+}}_{\scriptscriptstyle{#1}}}
}
\newcommand{\hJacSymb}{\mathbf{\hat{J}}}
\newcommand{\hJac}[1][]
{
  \ifthenelse{\equal{#1}{}}
  {\hJacSymb}
  {{\hJacSymb}_{\scriptscriptstyle{#1}}}
}
\newcommand{\matWSymb}{\mathbf{W}}
\newcommand{\matW}[1][]
{
  \ifthenelse{\equal{#1}{}}
  {\matWSymb}
  {{\matWSymb}_{\scriptscriptstyle{#1}}}
}

\newcommand{\PrincipalK}[1][]
{
  \ifthenelse{\equal{#1}{}}
  {k}
  {k_{#1}}
}

\newcommand{\vecsymb}{q}
\newcommand{\vecFun}[1][]{
  \ifthenelse{\equal{#1}{}}
  {\Vector{\vecsymb}}
  {\vecsymb^{#1}}
}
\newcommand{\tvecFun}[1][]{
  \ifthenelse{\equal{#1}{}}
  {\tilde{\vecsymb}}
  {\tilde{\vecsymb}_{#1}}
}

\newcommand{\vvsymb}{v}
\newcommand{\vv}[1][]{
  \ifthenelse{\equal{#1}{}}
  {\mathbf{\vvsymb}}
  {\vvsymb^{#1}}
}
\newcommand{\wwsymb}{w}
\newcommand{\ww}[1][]
{
  \ifthenelse{\equal{#1}{}}
  {\mathbf{\wwsymb}}
  {\wwsymb^{#1}}
}
\newcommand{\uusymb}{u}
\newcommand{\uu}[1][]
{
  \ifthenelse{\equal{#1}{}}
  {\mathbf{\uusymb}}
  {\uusymb^{#1}}
}
\newcommand{\qqsymb}{q}
\newcommand{\qq}[1][]
{
  \ifthenelse{\equal{#1}{}}
  {\mathbf{\qqsymb}}
  {\qqsymb^{#1}}
}

\newcommand{\VecFieldSymbol}{X}
\newcommand{\VecField}[1][]
{
  \ifthenelse{\equal{#1}{}}
  {\VecFieldSymbol}
  {\VecFieldSymbol^{#1}}
}
\newcommand{\VecFieldYSymbol}{Y}
\newcommand{\VecFieldY}[1][]
{
  \ifthenelse{\equal{#1}{}}
  {\VecFieldYSymbol}
  {\VecFieldYSymbol^{#1}}
}

\newcommand{\xvsymb}{x}
\newcommand{\xv}[1][]
{
  \ifthenelse{\equal{#1}{}}
  {\mathbf{\xvsymb}}
  {\mathbf{\xvsymb}_{\scriptscriptstyle{#1}}}
}
\newcommand{\xvcomp}[1][]{
  \ifthenelse{\equal{#1}{}}
  {\xvsymb}
  {\xvsymb^{\scriptscriptstyle{#1}}}
}
\newcommand{\xcg}[1][]{
  \ifthenelse{\equal{#1}{}}
  {\xvcomp[1]}
  {\xvcomp[1]_{\scriptscriptstyle{#1}}}
}
\newcommand{\ycg}[1][]{
  \ifthenelse{\equal{#1}{}}
  {\xvcomp[2]}
  {\xvcomp[2]_{\scriptscriptstyle{#1}}}
}
\newcommand{\zcg}[1][]{
  \ifthenelse{\equal{#1}{}}
  {\xvcomp[3]}
  {\xvcomp[3]_{\scriptscriptstyle{#1}}}
}
\newcommand{\gcscomp}{\xcg,\ycg,\zcg}

\newcommand{\hxvsymb}{\hat{x}}
\newcommand{\hxv}[1][]
{
  \ifthenelse{\equal{#1}{}}
  {\mathbf{\hxvsymb}}
  {\mathbf{\hxvsymb}_{\scriptscriptstyle{#1}}}
}
\newcommand{\hxvcomp}[1][]{
  \ifthenelse{\equal{#1}{}}
  {\hxvsymb}
  {\hxvsymb^{\scriptscriptstyle{#1}}}
}
\newcommand{\hxcg}[1][]{
  \ifthenelse{\equal{#1}{}}
  {\hxvcomp[1]}
  {\hxvcomp[1]_{\scriptscriptstyle{#1}}}
}
\newcommand{\hycg}[1][]{
  \ifthenelse{\equal{#1}{}}
  {\hxvcomp[2]}
  {\hxvcomp[2]_{\scriptscriptstyle{#1}}}
}
\newcommand{\hzcg}[1][]{
  \ifthenelse{\equal{#1}{}}
  {\hxvcomp[3]}
  {\hxvcomp[3]_{\scriptscriptstyle{#1}}}
}
\newcommand{\hgcscomp}{\hxcg,\hycg}

\newcommand{\svsymb}{s}
\newcommand{\sv}[1][]{
  \ifthenelse{\equal{#1}{}}
  {\mathbf{\svsymb}}
  {\mathbf{\svsymb}_{\scriptscriptstyle{#1}}}
}
\newcommand{\svcomp}[1][]
{
  \ifthenelse{\equal{#1}{}}
  {\svsymb}
  {\svsymb^{\scriptscriptstyle{#1}}}
}
\newcommand{\xcl}[1][]{
  \ifthenelse{\equal{#1}{}}
   {\svcomp[1]}
   {\svcomp[1]_{\scriptscriptstyle{#1}}}
}
\newcommand{\ycl}[1][]{
  \ifthenelse{\equal{#1}{}}
   {\svcomp[2]}
   {\svcomp[2]_{\scriptscriptstyle{#1}}}
}
\newcommand{\zcl}[1][]{
  \ifthenelse{\equal{#1}{}}
   {\svcomp[3]}
   {\svcomp[3]_{\scriptscriptstyle{#1}}}
}
\newcommand{\lcscomp}{\xcl,\ycl}

\newcommand{\ProjSymb}{\operatorname{\pi}}
\newcommand{\ProjFun}[2][]{
  \ifthenelse{\equal{#1}{}}
  {\ProjSymb\left(#2\right)}
  {\ProjSymb_{\scriptscriptstyle{#1}}\left(#2\right)}
}
\newcommand{\Prm}[1][]
{
  \ifthenelse{\equal{#1}{}}
  {\operatorname{pr}}
  {\operatorname{pr}_{\scriptscriptstyle{#1}}}
}
\newcommand{\TanPlane}[2][]
{
  \ifthenelse{\equal{#1}{}}
  {T_{\scriptscriptstyle{\point}}#2}
  {T_{\scriptscriptstyle{#1}}#2}
}

\newcommand{\SubsetSymbol}{\mathcal{U}}

\newcommand{\SubsetU}[1][]
{
  \ifthenelse{\equal{#1}{}}
  {{U}}
  {{U}_{#1}}
}
\newcommand{\SubsetV}[1][]
{
  \ifthenelse{\equal{#1}{}}
  {{V}}
  {{V}_{#1}}
}
\newcommand{\SubsetWsymb}{K}
\newcommand{\SubsetW}[1][]
{
  \ifthenelse{\equal{#1}{}}
  {{\SubsetWsymb}}
  {{\SubsetWsymb}_{#1}}
}
\newcommand{\NeighSymbol}{\mathcal{N}}
\newcommand{\Neigh}[1][]
{
  \ifthenelse{\equal{#1}{}}
  {\NeighSymbol_{\point}}
  {\NeighSymbol_{#1}}
}
\newcommand{\NeighSurf}[1][]
{
  \ifthenelse{\equal{#1}{}}
  {\SubsetSymbol_{\point}}
  {\SubsetSymbol_{#1}}
}
\newcommand{\NormSymb}{\mathbf{N}}

\newcommand{\normalSurf}[1][]
{
  \ifthenelse{\equal{#1}{}}
  {\NormSymb}
  {\NormSymb(#1)}
}
\newcommand{\normalInterp}[1][]
{
  \ifthenelse{\equal{#1}{}}
   {\tilde{\NormSymb}}
   {\tilde{\NormSymb}_{\scriptscriptstyle{#1}}}
}

\newcommand{\normalEdge}{\mathbf{\nu}} 
\newcommand{\basisCC}{t}
\newcommand{\basisGC}{e}
\newcommand{\vecBaseGC}[1][]
{
  \ifthenelse{\equal{#1}{}}
  {\mathbf{\basisGC}}
  {\mathbf{\basisGC}_{#1}}
}
\newcommand{\vecBasePhys}[1][]
{
  \ifthenelse{\equal{#1}{}}
  {\mathbf{\basisGC}}
  {\mathbf{\basisGC}_{#1}}
}

\newcommand{\vecBaseCCcv}[1][]
{
  \ifthenelse{\equal{#1}{}}
  {\mathbf{\basisCC}}
  {\mathbf{\basisCC}_{#1}}
}

\newcommand{\tvecBaseCCcv}[1][]
{
  \ifthenelse{\equal{#1}{}}
  {\tilde{\mathbf{\basisCC}}}
  {\tilde{\mathbf{\basisCC}}_{#1}}
}
\newcommand{\hvecBaseCCcv}[1][]
{
  \ifthenelse{\equal{#1}{}}
  {\hat{\mathbf{\basisCC}}}
  {\hat{\mathbf{\basisCC}}_{#1}}
}
\newcommand{\vecBaseCCctrv}[1][]
{
  \ifthenelse{\equal{#1}{}}
  {\mathbf{\basisCC}}
  {\mathbf{\basisCC}^{#1}}
}

\newcommand{\metricsymbol}{g}
\newcommand{\metrTensCv}[1]{\metricsymbol_{\mbox{\tiny{#1}}}}

\newcommand{\metrTensCtrv}[1]{\metricsymbol^{#1}}

\newcommand{\first}[1]{
  \IfEqCase{#1}{
    {1}{\operatorname{E}}
    {2}{\operatorname{F}}
    {3}{\operatorname{G}}
  }
  [\PackageError{first}{Undefined option to first: #1}{}]%
}
\newcommand{\SecondFormSymbol}{\ensuremath{\operatorname{II}}}
\newcommand{\SecondForm}[1][]
{
  \ifthenelse{\equal{#1}{}}
  {\SecondFormSymbol_{\point}}
  {\SecondFormSymbol_{#1}}
}
\newcommand{\second}[1]{
  \IfEqCase{#1}{
    {1}{\operatorname{e}}
    {2}{\operatorname{f}}
    {3}{\operatorname{g}}
  }
  [\PackageError{first}{Undefined option to first: #1}{}]
}
\newcommand{\WeigSymbol}{\mathcal{W}}
\newcommand{\Weig}[1][]
{
  \ifthenelse{\equal{#1}{}}
  {\WeigSymbol}
  {\WeigSymbol_{#1}}
}

\newcommand{\velSymbol}{u}
\newcommand{\vectvel}[1][]
{
   \ifthenelse{\equal{#1}{}}
   {\vec{\velSymbol}}
   {\vec{\velSymbol}(#1)}
}

\newcommand{\velcompContr}[2][i]
{
   \ifthenelse{\equal{#2}{}}
   {\velSymbol^{#1}}
   {\velSymbol^{#1}(#2)}}
\newcommand{\velcompPhys}[2][i]
{
   \ifthenelse{\equal{#2}{}}
   {\velSymbol_{(#1)}}
   {\velSymbol_{(#1)}(#2)}}

\newcommand{\velSymbolRP}{v}
\newcommand{\velcompContrRP}[2][i]
{
   \ifthenelse{\equal{#2}{}}
   {\velSymbolRP^{#1}}
   {\velSymbolRP^{#1}(#2)}}
\newcommand{\velRP}[1][]
{
  \ifthenelse{\equal{#1}{}}
  {\velSymbolRP}
  {\velSymbolRP_{#1}}
}

\newcommand{\velcompApprox}[2][i]
{
   \ifthenelse{\equal{#2}{}}
   {\velSymbol^{#1)}}
   {\velSymbol^{#1}_{(#2)}}
}

\newcommand{\VelSymbol}{U}
\newcommand{\vectVel}[1][]
{
   \ifthenelse{\equal{#1}{}}
   {\vec{\VelSymbol}}
   {\vec{\VelSymbol}(#1)}
}
\newcommand{\Velcomp}[2][i]
{
   \ifthenelse{\equal{#2}{}}
   {\VelSymbol^{#1}}
   {\VelSymbol^{#1}(#2)}
}
\newcommand{\VprimoSymbol}{\tilde{u}}
\newcommand{\Vprimo}[1][]
{
   \ifthenelse{\equal{#1}{}}
   {\VprimoSymbol}
   {\VprimoSymbol(#1)}
}
\newcommand{\VprimoComp}[2][i]
{
   \ifthenelse{\equal{#2}{}}
   {\VprimoSymbol^{#1}}
   {\VprimoSymbol^{#1}(#2)}
}
\newcommand{\ttvelSymbol}{\tilde{\Mymathbb{u}}}
\newcommand{\ttvel}[1][]
{
   \ifthenelse{\equal{#1}{}}
   {\mathbf{\ttvelSymbol}}
   {\mathbf{\ttvelSymbol}(#1)}
}
\newcommand{\ttvelComp}[2][i]
{
   \ifthenelse{\equal{#2}{}}
   {\ttvelSymbol^{#1}}
   {\ttvelSymbol^{#1}(#2)}
}

\newcommand{\MatAlphaSymbol}{\mathbb{A}}
\newcommand{\MatAlpha}[1][]{%
  \ifthenelse{\equal{#1}{}}
  {\MatAlphaSymbol}
  {\MatAlphaSymbol_{#1}}
}

\newcommand{\QSymbol}{q}
\newcommand{\Qdisch}[1][]
{
   \ifthenelse{\equal{#1}{}}
   {\mathbf{\QSymbol}}
   {\mathbf{\QSymbol}(#1)}
}
\newcommand{\Qcomp}[2][i]
{
   \ifthenelse{\equal{#2}{}}
   {\QSymbol^{#1}}
   {\QSymbol^{#1}(#2)}
}
\newcommand{\Qvect}[1][]
{
   \ifthenelse{\equal{#1}{}}
   {\mathbf{\QSymbol}}
   {\mathbf{\QSymbol}(#1)}
}
\newcommand{\FricSymbol}{f}
\newcommand{\vectFric}[1][]
{
   \ifthenelse{\equal{#1}{}}
   {\mathbf{\FricSymbol}}
   {\mathbf{\FricSymbol}_{\scriptscriptstyle{#1}}}
}
\newcommand{\Friccomp}[2][i]
{
   \ifthenelse{\equal{#2}{}}
   {\FricSymbol_{#1}}
   {\FricSymbol_{#1}(#2)}}
\newcommand{\BFsymbol}{\tau}
\newcommand{\BottomFriction}[1][]{
  \ifthenelse{\equal{#1}{}}
  {\BFsymbol_{b}}
  {\BFsymbol_{b}^{#1}}
}

\newcommand{\ProjMatSymb}{\mathbb{P}}
\newcommand{\ProjMat}[1][]{
  \ifthenelse{\equal{#1}{}}
  {\ProjMatSymb}
  {\ProjMatSymb_{#1}}
}
\newcommand{\IDSymbol}{\mathbb{I}}
\newcommand{\IDtens}[1][]{
  \ifthenelse{\equal{#1}{}}
  {\IDSymbol}
  {\IDSymbol(#1)}
}
\newcommand{\IDMat}{\IDSymbol}

\newcommand{\tensSymbol}{\mathbb{T}}
\newcommand{\tenscompSymbol}{\tau}
\newcommand{\tens}[1][]{
  \ifthenelse{\equal{#1}{}}
  {\tensSymbol}
  {\tensSymbol(#1)}
}
\newcommand{\tenscomp}[2][ij]
{
  \ifthenelse{\equal{#2}{}}
  {\tenscompSymbol^{#1}}
  {\tenscompSymbol^{#1}(#2)}
}
\newcommand{\tensrow}[2][i]
{
  \ifthenelse{\equal{#2}{}}
  {\tensSymbol^{(#1)}}
  {\tensSymbol^{(#1)}(#2)}
}
\newcommand{\TensSymbol}{\mathbf{T}}
\newcommand{\Tens}[1][]{
  \ifthenelse{\equal{#1}{}}
  {\TensSymbol}
  {\TensSymbol_{#1}}
}

\newcommand{\TensCompSymbol}{\TensSymbol}
\newcommand{\TensComp}[2][ij]
{
  \ifthenelse{\equal{#2}{}}
  {\TensCompSymbol^{#1}}
  {\TensCompSymbol^{#1}(#2)}
}

\newcommand{\tensPrimoSymbol}{\tilde{\mathbf{\tau}}}
\newcommand{\tensPrimo}[1][]{
  \ifthenelse{\equal{#1}{}}
  {\tensPrimoSymbol}
  {\tensPrimoSymbol(#1)}
}
\newcommand{\tensPrimoCompSymbol}{\tensPrimoSymbol}
\newcommand{\tensPrimoComp}[2][ij]
{
  \ifthenelse{\equal{#2}{}}
  {\tensPrimoCompSymbol^{#1}}
  {\tensPrimoCompSymbol^{#1}(#2)}
}
\newcommand{\MCxl}[1][]{\ifthenelse{\equal{#1}{}}{h_{(1)}}{h_{(1),#1}}} 
\newcommand{\MCyl}[1][]{\ifthenelse{\equal{#1}{}}{h_{(2)}}{h_{(2),#1}}} 
\newcommand{\MCzl}[1][]{\ifthenelse{\equal{#1}{}}{h_{(3)}}{h_{(3),#1}}}

\newcommand{\MPsymb}{h}
\newcommand{\Lev}{\ell}
\newcommand{\smallestH}[1][]
{
  \ifthenelse{\equal{#1}{}}
    {{l}}
    {{l}_{#1}}
}
\newcommand{\meshparam}[1][]
{
  \ifthenelse{\equal{#1}{}}
    {\MPsymb}
    {\MPsymb_{\scriptscriptstyle{#1}}}
}
\newcommand{\InradiusSymbol}{r}
\newcommand{\Inradius}[1][]
{
  \ifthenelse{\equal{#1}{}}
  {\InradiusSymbol}
  {\InradiusSymbol_{\scriptscriptstyle{#1}}}
}
\newcommand{\Tsymb}{\mathcal{T}}
\newcommand{\Triang}[1][]
{
  \ifthenelse{\equal{#1}{}}
    {\Tsymb}
    {\Tsymb_{#1}}
}
\newcommand{\TriangH}[1][]
{
  \ifthenelse{\equal{#1}{}}
    {\Tsymb_{\meshparam}}
    {\Tsymb_{#1}}
}
\renewcommand{\Triang}{\TriangH}
\newcommand{\Edgesymb}{\sigma}
\newcommand{\Edge}[1][]{
  \ifthenelse{\equal{#1}{}}
    {\Edgesymb}
    {\Edgesymb_{#1}}
}
\newcommand{\EdgeH}[1][]{
  \ifthenelse{\equal{#1}{}}
    {\Edgesymb_{\meshparam}}
    {\Edgesymb_{\meshparam,#1}}
}
\newcommand{\NEdge}[1][]{
  \ifthenelse{\equal{#1}{}}
    {N_{\Edgesymb}}
    {N_{\Edgesymb({#1})}}
}
\newcommand{\Cellsymb}{T}
\newcommand{\Cell}[1][]{
  \ifthenelse{\equal{#1}{}}
    {\Cellsymb}
    {\Cellsymb_{#1}}
}
\newcommand{\tCell}[1][]{
  \ifthenelse{\equal{#1}{}}
    {\tilde{\Cellsymb}}
    {\tilde{\Cellsymb}_{#1}}
}
\newcommand{\CellH}[1][]{
  \ifthenelse{\equal{#1}{}}
    {\Cellsymb_{\meshparam}}
    {\Cellsymb_{\meshparam,#1}}
}
\newcommand{\areaSymb}{\mathcal{A}}
\newcommand{\CellArea}[1][]
{
  \ifthenelse{\equal{#1}{}}
    {\areaSymb_{\Cell}}
    {\areaSymb_{#1}}
}
\newcommand{\CellHArea}[1][]
{
  \ifthenelse{\equal{#1}{}}
    {\areaSymb_{\CellH}}
    {\areaSymb_{\meshparam,#1}}
}

\newcommand{\NCell}[1][]{
  \ifthenelse{\equal{#1}{}}
    {N_{\Cellsymb}}
    {N_{\Cellsymb({#1})}}
}

\newcommand{\lengthSymb}{{\ell}}
\newcommand{\edgeLength}[1][]
{
  \ifthenelse{\equal{#1}{}}
  {\lengthSymb_{\Edge}}
  {\lengthSymb_{#1}}
}
 
\newcommand{\edgeHLength}[1][]
{
  \ifthenelse{\equal{#1}{}}
  {\lengthSymb_{\EdgeH}}
  {\lengthSymb_{\meshparam,#1}}
}

\newcommand{\Sourcesymb}{\mathbf{S}}
\newcommand{\Source}[1][]
{
  \ifthenelse{\equal{#1}{}}
    {\Sourcesymb}
    {\Sourcesymb_{#1}}
}

\newcommand{\SourceEdge}[1][]{
  \ifthenelse{\equal{#1}{}}
    {\Sourcesymb_{ij}}
    {\Sourcesymb_{#1}}
}

\newcommand{\FluxEdgesymb}{\mathbf{F}}
\newcommand{\FluxEdge}[1][]{
  \ifthenelse{\equal{#1}{}}
    {\FluxEdgesymb_{ij}}
    {\FluxEdgesymb_{#1}}
}

\newcommand{\numFlux}[1][]
{
  \ifthenelse{\equal{#1}{}}
  {\tilde{\fsymb}}
  {\tilde{\fsymb}_{#1}}
}

\newcommand{\FluxFuncNormSymbol}{\mathbf{F}}
\newcommand{\FluxFuncNorm}[1][]{
  \ifthenelse{\equal{#1}{}}
    {\FluxFuncNormSymbol^{\normalEdge}}
    {\FluxFuncNormSymbol^{\normalEdge}_{#1}}    
}

\newcommand{\JacobianSymbol}{\mathbf{A}}
\newcommand{\Jacobian}[1][]{
  \ifthenelse{\equal{#1}{}}
    {\JacobianSymbol}
    {\JacobianSymbol_{#1}}    
}
\newcommand{\EValSymbol}{\lambda}
\newcommand{\EVal}[1][]{
  \ifthenelse{\equal{#1}{}}
  {\EValSymbol}
  {\EValSymbol_{#1}}
}
\newcommand{\EVecSymbol}{\mathbf{r}}
\newcommand{\EVec}[2][]{
  \ifthenelse{\equal{#1}{}}
  {\EVecSymbol^{(#2)}}
  {\EVecSymbol^{(#2)}_{#1}}
}

\newcommand{\midPointEdge}[1][]
{
  \ifthenelse{\equal{#1}{}}
  {\midPoint_{\scriptscriptstyle\Edge}}
  {\midPoint_{\scriptscriptstyle\Edge[#1]}}
}
\newcommand{\gpPointEdgeDG}[1][]
{
  \ifthenelse{\equal{#1}{}}
  {\point_{\scriptscriptstyle\Edge}}
  {\point_{\scriptscriptstyle\Edge,#1}}
}
\newcommand{\midPointCell}[1][]
{
  \ifthenelse{\equal{#1}{}}
  {\midPoint_{\scriptscriptstyle\Cell}}
  {\midPoint_{\scriptscriptstyle\Cell[#1]}}
}

\newcommand{\speedRS}[2][]
{
  \ifthenelse{\equal{#2}{}}
  {S_{#2}}
  {S_{#2}^{#1}}
}
\newcommand{\Interpolant}{I_{\meshparam}}

\newcommand{\ContSymbol}{C}
\newcommand{\Cont}[1][]{
  \ifthenelse{\equal{#1}{}}
  {\ContSymbol^{0}}
  {\ContSymbol^{#1}}
}
\newcommand{\Cinf}[1][]{
  \ifthenelse{\equal{#1}{}}
  {\ContSymbol^{\infty}}
  {\ContSymbol^{\infty}(#1)}
}
\newcommand{\SobSymbol}{W}
\newcommand{\HilbSymbol}{H}
\newcommand{\Sob}[2][]{
  \ifthenelse{\equal{#1}{}}
  {\HilbSymbol^{#2}}
  {\SobSymbol^{#2,#1}}  
}

\newcommand{\LspaceSymb}{L}
\newcommand{\Lspace}[1][]{
  \ifthenelse{\equal{#1}{}}
  {\LspaceSymb^{2}}
  {\LspaceSymb^{#1}}  
}

\newcommand{\TestSpSymbol}{\mathcal{V}}
\newcommand{\TestSpace}[1][]{
  \ifthenelse{\equal{#1}{}}
  {\TestSpSymbol({\SurfDomain})}
  {\TestSpSymbol_{#1}({\SurfDomain})}
}
\newcommand{\TestSpaceEmbedded}[1][]{
  \ifthenelse{\equal{#1}{}}
  {\TestSpSymbol(\TriangH(\SurfDomain))}
  {\TestSpSymbol_{#1}(\TriangH(\SurfDomain))}
}
\newcommand{\TestSpaceIntrinsic}[1][]{
  \ifthenelse{\equal{#1}{}}
  {\TestSpSymbol(\Triang(\SurfDomain))}
  {\TestSpSymbol_{#1}(\Triang(\SurfDomain))}
}
\newcommand{\TestSpaceChart}[1][]{
  \ifthenelse{\equal{#1}{}}
  {\TestSpSymbol(\Triang(\SubsetU))}
  {\TestSpSymbol_{#1}(\Triang(\SubsetU))}
}
\newcommand{\TestSpaceCell}[1][]{
  \ifthenelse{\equal{#1}{}}
  {\TestSpSymbol_{\meshparam}(\Cell)}
  {\TestSpSymbol_{\meshparam}(\Cell_{#1})}
}

\newcommand{\TestSpaceApprox}{\TestSpaceIntrinsic[\meshparam]}

\newcommand{\PolySymb}{\mathcal{P}}

\newcommand{\PONE}{\PolySymb_{1}}

\newcommand{\GaussCurv}{\curvature}

\newcommand{\AngleSymbol}{\theta}
\newcommand{\DevAngle}[1][]
{
  \ifthenelse{\equal{#1}{}}
  {\AngleSymbol}
  {\AngleSymbol_{\scriptscriptstyle{#1}}}
}
\newcommand{\relheightSymb}{\pi}
\newcommand{\relheight}[1][]
{
  \ifthenelse{\equal{#1}{}}
  {\relheightSymb_{\scriptscriptstyle{\SurfDomain}}}
  {\relheightSymb_{\scriptscriptstyle{#1}}}
}

\newcommand{\conormal}{\mathbf{\mu}}

\newcommand{\SolSymbol}{u}
\newcommand{\Sol}{\SolSymbol}
\newcommand{\AvgSol}{\bar{\SolSymbol}}

\newcommand{\force}{f}

\newcommand{\DiffCoef}{\epsilon}

\newcommand{\DiffTens}{\mathbb{D}}
\newcommand{\DiffEig}[1][]
{ \ifthenelse{\equal{#1}{}}
  {d}
  {d_{#1}}
}
\newcommand{\DiffEigMax}{\DiffEig^*}
\newcommand{\DiffEigMin}{\DiffEig_*}

\newcommand{\BilinearStiffSymbol}{a}
\newcommand{\BilinearStiff}[3][]
{
  \ifthenelse{\equal{#1}{}}
  {\BilinearStiffSymbol(#2,#3)}    
  {\BilinearStiffSymbol_{#1}(#2,#3)}    
}
\newcommand{\BilinearAdvSymbol}{b}
\newcommand{\BilinearAdv}[3][]
{
  \ifthenelse{\equal{#1}{}}
  {\BilinearAdvSymbol(#2,#3)}    
  {\BilinearAdvSymbol_{#1}(#2,#3)}    
}
\newcommand{\BilinearMassSymbol}{m}
\newcommand{\BilinearMass}[3][]
{
  \ifthenelse{\equal{#1}{}}
  {\BilinearMassSymbol(#2,#3)}    
  {\BilinearMassSymbol_{#1}(#2,#3)}    
}
\newcommand{\BilinearReactSymbol}{c}
\newcommand{\BilinearReact}[3][]
{
  \ifthenelse{\equal{#1}{}}
  {\BilinearReactSymbol(#2,#3)}    
  {\BilinearReactSymbol_{#1}(#2,#3)}    
}

\newcommand{\RhsSymbol}{F}
\newcommand{\Rhs}[1]{\RhsSymbol(#1)}

\newcommand{\TestSymbol}{v}
\newcommand{\Test}[1][]
{
  \ifthenelse{\equal{#1}{}}
  {\TestSymbol}  
  {\TestSymbol_{\scriptscriptstyle{#1}}}
}

\newcommand{\SolApprox}{\Sol_{\meshparam}}

\newcommand{\forceApprox}{\force_{\meshparam}}

\newcommand{\hnodalBasis}[1]{\hat{\varphi}_{\scriptscriptstyle{#1}}}
\newcommand{\nodalBasis}[1]{\varphi_{\scriptscriptstyle{#1}}}

\newcommand{\nNodes}[1][]
{
  \ifthenelse{\equal{#1}{}}
  {N^{\scriptscriptstyle{dof}}}
  {N^{\scriptscriptstyle{dof}}_{\scriptscriptstyle{#1}}}
}

\newcommand{\TestApprox}[1][]
{
  \ifthenelse{\equal{#1}{}}
  {\TestSymbol_{\scriptscriptstyle{\meshparam}}}  
  {\TestSymbol_{\scriptscriptstyle{\meshparam,#1}}}
}
\newcommand{\bTestApprox}[1][]
{
  \ifthenelse{\equal{#1}{}}
  {\hat{\TestSymbol}_{\scriptscriptstyle{\meshparam}}}  
  {\hat{\TestSymbol}_{\scriptscriptstyle{\meshparam,#1}}}
}
\newcommand{\TestWSymbol}{w}
\newcommand{\TestW}[1][]
{
  \ifthenelse{\equal{#1}{}}
  {\TestWSymbol}  
  {\TestWSymbol_{\scriptscriptstyle{#1}}}
}

\newcommand{\TestWApprox}[1][]
{
  \ifthenelse{\equal{#1}{}}
  {\TestWSymbol_{\scriptscriptstyle{\meshparam}}}  
  {\TestWSymbol_{\scriptscriptstyle{\meshparam,#1}}}
}
\newcommand{\bTestWApprox}[1][]
{
  \ifthenelse{\equal{#1}{}}
  {\hat{\TestWSymbol}_{\scriptscriptstyle{\meshparam}}}  
  {\hat{\TestWSymbol}_{\scriptscriptstyle{\meshparam,#1}}}
}

\newcommand{\coeffTest}[1]{\TestSymbol_{#1}}

\newcommand{\BilinearStiffApprox}[2]{\BilinearStiffSymbol_{\meshparam}(#1,#2)}

\newcommand{\RhsApprox}[1]{\RhsSymbol_{\meshparam}(#1)}

\newcommand{\matrixMass}{\mathbf{M}}

\newcommand{\viscositySymb}{\nu}
\newcommand{\viscosity}[1][]
{
  \ifthenelse{\equal{#1}{}}
  {\viscositySymb}
  {\viscositySymb_{\scriptscriptstyle{#1}}}
}

\newcommand{\ResidualSymbol}{R}
\newcommand{\Residual}[1][]
{
  \ifthenelse{\equal{#1}{}}
  {\ResidualSymbol}
  {\ResidualSymbol_{#1}}
}

\newcommand{\CPoincare}{C_{_{\SurfDomain}}}

\newcommand{\Riccibound}{R}
\newcommand{\curvature}[1][]
{
  \ifthenelse{\equal{#1}{}}
  {\kappa}
  {\kappa_{#1}}
}
\newcommand{\Eig}{\lambda_1}

\newcommand{\QuadRule}[2][]
{
  \ifthenelse{\equal{#1}{}}
  {Q(#2)}
  {Q_{#1}(#2)}
}     